\documentclass[12pt, twoside]{article}
\usepackage{amsmath,amsthm,amssymb}
\usepackage{times}
\usepackage{enumerate}

\def\titlerunning#1{\gdef\titrun{#1}}
\makeatletter
\def\author#1{\gdef\autrun{\def\and{\unskip, }#1}\gdef\@author{#1}}
\def\address#1{{\def\and{\\\hspace*{18pt}}\renewcommand{\thefootnote}{}%
\footnote {#1}}%
\markboth{\autrun}{\titrun}}
\makeatother
\def\email#1{e-mail: #1}
\def\subjclass#1{{\renewcommand{\thefootnote}{}%
\footnote{\emph{Mathematics Subject Classification (2010):} #1}}}
\def\keywords#1{\par\medskip
\noindent\textbf{Keywords.} #1}

\newtheorem{thm}{Theorem}[section]

\theoremstyle{definition}

\newtheorem{rem}[thm]{Remark}

\numberwithin{equation}{section}

\usepackage{graphicx}
\usepackage{subcaption}
\usepackage{caption}
\usepackage{xcolor}
\usepackage{dsfont}
\usepackage{amssymb}
\usepackage{tikz}

\newtheorem{prop}[thm]{Proposition}

\def\R{\mathbb{R}}
\def\Z{\mathbb{Z}}
\def\LL{\mathcal L}

\def\range{\mathrm{range}\,}
\def\dd{\mathrm{d}}

\def\({\left(}
\def\){\right)}
\def\nullo{\mathrm{null}\,}

\def\form{\mathcal C}

\DeclareMathOperator{\e}{e}
\DeclareMathOperator{\supp}{supp\,}

\tikzset{vertex/.style={circle,fill=black,inner sep=2pt},
ctVertex/.style={diamond,fill=black,inner sep=2pt},
bigvertex/.style={circle,fill=black,inner sep=4pt},
smallvertex/.style={circle,fill=black,inner sep=1pt},
E/.append style={fill=white,draw},
probeEP/.style={circle,fill=black,draw,inner sep=2pt,
  prefix after command= {\pgfextra{\tikzset{every pin/.style = {pin edge={decorate,decoration={snake,amplitude=2pt,segment length =4pt}}}}}}
},
bareProbeEP/.style={rectangle,fill=black,draw,inner sep=3pt,
  prefix after command= {\pgfextra{\tikzset{every pin/.style = {pin edge={decorate,decoration={snake,amplitude=2pt,segment length =4pt}}}}}}
},
nuEP/.style={circle,fill=white,draw, inner sep=2pt},
linelabel/.style={sloped,above,very near start, inner sep=1pt,execute at begin node=$\scriptstyle,execute at end node=$},
baseline=(current  bounding  box.center),doubled/.style={double distance= 1pt,line width=1.5pt}}

\begin{document}


\baselineskip=17pt


\titlerunning{LRO in atomistic models for solids}

\title{Long range order in atomistic models for solids}

\author{Alessandro Giuliani
\and 
Florian Theil}

\date{}

\maketitle

\address{A. Giuliani: Universit\`{a} degli Studi Roma Tre, Dipartimento di Matematica e Fisica, L.go S. L. Murialdo 1, 00146 Roma, Italy; and: Centro Linceo Interdisciplinare {\it Beniamino Segre}, Accademia Nazionale dei Lincei, Palazzo Corsini, Via della Lungara 10,
00165 Roma, Italy; \email{giuliani@mat.uniroma3.it}
\and
F. Theil: Warwick University, Mathematics Institute, Coventry, CV47AL, UK;
\email{f.theil@warwick.ac.uk}}



\begin{abstract}
The emergence of long-range order at low temperatures in atomistic systems with continuous symmetry is a fundamental, yet poorly understood phenomenon in Physics. To address this challenge we study a discrete microscopic model for an elastic crystal with dislocations in three dimensions, previously introduced by Ariza and Ortiz. The model is rich enough to support some realistic features of three-dimensional dislocation theory, most notably grains and the Read-Shockley law for grain boundaries, which we
rigorously derive in a simple, explicit, geometry.
We analyze the model at positive temperatures, in terms of a Gibbs distribution with energy function given by the Ariza-Ortiz Hamiltonian plus a contribution from the dislocation cores.
Our main result is that the model exhibits long range positional order at low temperatures. The proof is based on the tools of discrete exterior calculus, together with cluster expansion techniques.

\end{abstract}

\section{Introduction}

The derivation of the low temperature properties of crystalline solids, starting from a microscopic, atomistic, model, represents a formidable challenge both for theoreticians and
practitioners. Realistic atomistic models for solids are characterized by the invariance under the Euclidean symmetries of translations and rotations, which are supposedly broken at low temperatures,
as the very existence of crystals in nature witnesses. Unfortunately, from a mathematical point of view, our understanding of the phenomenon of continuous symmetry breaking is still quite limited and, as a consequence,
the mathematical theory of crystalline solids is still in a primitive stage. Even at zero temperature, there are only limited results on the ground state structure of the system: in particular, there are only
few, highly simplified, atomistic models for which one can rigorously prove that the ground state is periodic \cite{BPT14, EL09, FT15, T06}. Even less is known at positive temperatures where most rigorous results are restricted to lattice systems, e.g. \cite{AKM16, G19, HMR14}. A notable exception is \cite{Au15} which establishes the existence of orientational order in a particle system without lattice structure.

Heuristically, we expect that the low energy physics of crystalline materials is dominated by dislocations defects, which interact among each other via an electrostatic-like interaction, and
by the formation of grains, which correspond to portions of the crystal with some fixed rotation relative to a background orientation. The grain boundaries are collections of dislocations that are geometrically necessary
to connect differently oriented lattices. Remarkably, even though isolated dislocations interact among each other via a Coulomb-like interaction, the energy of a grain appears to scale like the size of its boundary.
For a recent mathematical account of this phenomenon, see~\cite{LL17}.

There is significant literature on continuum theories for dislocations, see \cite{HL82} for a starting point. Typically dislocations are represented as closed loops, the energy of a single dislocation loop is proportional to its length,
\cite{GM06}. Discrete dislocation line dynamics represent a very popular simulation technique for studying plasticity since the early 1990s, see e.g. \cite{BC06} and \cite{H18} for a recent account of mathematical results. Continuum models for dislocation configurations have been studied successfully within the framework of $\Gamma$-convergence, see e.g. \cite{CGM11, GLM10, GM05, GM06}. However, very few results are available on the microscopic derivation of effective continuum theories for dislocations or grain boundaries, see \cite{FPP18, LL17}.

Note that
macroscopic effects like plasticity or grain boundary motion are strongly temperature dependent: therefore, it is of particular interest to develop a thermal theory of dislocations, including an equilibrium theory based on the Gibbs distribution.

\medskip

In this paper, we consider a simple atomistic model for crystalline solids, previously introduced by Ariza and Ortiz \cite{ao05}. The Ariza-Ortiz model, even if highly simplified, possesses some realistic features expected in real solids,
which make it a good starting point for a quantitative understanding of the effects of dislocations and of the formations of grain boundaries. In particular, it has been used to perform discrete dislocation calculations of defects and grain
boundaries in graphene, see \cite{AO10, AOS10, AM16}. The Ariza-Ortiz model is a discrete model where the interaction energy depends not only on the positions of the particles, but also on the bond structure, see eq.\eqref{AOnoec} below for its precise definition. The model shares some analogies with the Villain model for rotators, in that the energy satisfies an exact additive decomposition property, which allows us to distinguish clearly the elastic (`spin wave') degrees of freedom, and those associated with dislocation defects, see also \cite{KT,NH,Y}. The simplicity of the model allows us to derive sharp estimates on the energy of the grains, on the one hand, and to rigorously characterize key properties of the equilibrium distribution of dislocations at positive temperatures, on the other.

Concerning the kinematics of the Ariza-Ortiz model, we confirm that it supports polycrystalline configuration with energy cost bounded from above by
the size of the grain boundary (Theorem~\ref{grain-scaling}). We also derive sharp asymptotic bounds, albeit in a simpler two-dimensional setting (Theorem~\ref{thm.RS}).
Our results confirm that the energy density of grain boundaries for small angles is consistent with the Read-Shockley law \cite{RS50}
\begin{equation}\label{readsh}\gamma(\theta) =\theta(c_0 - c_1 \log \theta) + o(\theta), \quad 0<\theta\ll 1, \end{equation}
where $\gamma(\theta)>0$ is the grain boundary energy density and $\theta$ is the orientation difference. See also \cite{LL17}, where the authors establish an upper bound consistent with the Read-Shockley law.

Concerning positive temperatures, we introduce a Gibbs distribution with energy function given by the Ariza-Ortiz Hamiltonian plus a contribution from the dislocation cores.
Our main result is that for low temperatures the system exhibits positional long-range order (Theorem~\ref{lro_simp.thm}).
In particular, this implies that polycrystalline configurations have low probability. To the best knowledge of the authors these are the first rigorous results on dislocations configurations at positive temperature in a microscopic,
atomistic, model. See also \cite{BCHMR18}, where similar results have been recently obtained in the context of a related mesoscopic model for crystalline solids.
The proof of long-range order is based on the strategy developed in \cite{fs82, KK86} for the three-dimensional $XY$ model and other lattice models with Abelian continuous symmetry. The key steps consist in: first, a
reduction of the model to an effective model for the dislocation defects, interacting via a tensorial analogue of the electrostatic force; second, a cluster expansion treatment of the latter. The computation of the
Green function characterizing the effective interaction among dislocations requires some care, in that the derivation must be compatible with the underlying symmetries of the system, most notably linearized rotational symmetry.
This is the key novel feature of the Ariza-Ortiz model, compared to other `scalar' models treated previously. In this part, we take advantage of the tools of exterior discrete calculus, some aspects of which we briefly review below,
for the reader's convenience.

\medskip

The paper is organized as follows. In Sect.\ref{secAO} we define the Ariza-Ortiz model and discuss its symmetries. In Sect.\ref{sec.3} we state our main results, first on the existence of long-range order at positive, low enough, temperatures, then on the energy scaling of grains and grain boundaries. In Sect.\ref{sec.DEC} we review a few selected aspects of exterior discrete calculus, required in the proofs of our main results. In Sect.\ref{proof_lro.sec},
we prove Theorem \ref{lro_simp.thm} on long-range positional order. In Sect.\ref{sec.RS}, we prove Theorem~\ref{thm.RS} on the asymptotic computation of the energy of a grain and derive the Read-Shockley law. Finally, in the appendices we collect a few technical results, including the explicit definition of the lattice cellular complex for the face centered cubic lattice, and the asymptotic computation of the correlation decay in the `spin wave approximation'.

\section{The Ariza-Ortiz model}\label{secAO}
Let $\LL \subset\R^3$ be the face-centered cubic (FCC) lattice, i.e. $\LL = \{ n_1 b_1 + n_2 b_2+ n_3 b_3 \; : \; n \in \Z^3\}$  where
\begin{equation}\label{b123} b_1 = \frac{1}{\sqrt{2}} \begin{pmatrix} 0\\1\\1\end{pmatrix},  b_2=\frac{1}{\sqrt{2}} \begin{pmatrix} 1\\0\\1\end{pmatrix}, b_3=\frac{1}{\sqrt{2}} \begin{pmatrix} 1\\1\\0\end{pmatrix},
\end{equation}
and let
\begin{equation}\label{LN}\Lambda=\Lambda^{(N)}=\{n_1b_1+n_2b_2+n_3b_3: \ n_i=\lfloor -N/2+1 \rfloor,\ldots,\lfloor N/2\rfloor,\quad i=1,2,3\} \subset \LL\end{equation}
be a finite box.
We will write $x\sim y$ if $x,y \in \LL$ are nearest neighbors, i.e. $|x-y| =1$. Note that each lattice point $x$ has exactly twelve nearest neighbors and $x\sim y$ if and only if $y=x\pm b_l$, with $l\in\{1,\ldots,6\}$, and
$b_1,b_2,b_3$ as in \eqref{b123}, $b_4:=b_3-b_2$, $b_5:=b_1-b_3$, $b_6:=b_2-b_1$.
The Hamiltonian is a quadratic form acting on pairs
$(u,\sigma)$
\begin{equation}\label{AOnoec} H_\mathrm{AO}(u,\sigma) = \frac{1}{2} \sum_{x\sim y} [(u(y)-u(x) -\sigma(x,y))\cdot (y-x)]^2.\end{equation}
where: the displacement $u: \LL\to \R^3$ satisfies Dirichlet boundary conditions on $\Lambda$, $u(x) = 0$ if $x \not \in \Lambda$; $\sigma:\{(x,y) \in \LL^2 \; : \; x\sim y\} \to \LL$ assigns a lattice-valued {\it slip} to each nearest neighbor pair $(x,y)$ and it also satisfies Dirichlet boundary conditions on $\Lambda$, that is, $\sigma(x,y)=0$ if $x,y\not\in\Lambda$. 
We assume that $\sigma(x,y)=-\sigma(y,x)$, so that the energy associated with a nearest neighbor pair $(x,y)$ is independent of the orientation. Moreover,
we let $\sum_{x\sim y}$ be the sum over the unordered pairs of nearest neighbor sites.
The interpretation of $\sigma$ is that
it accounts for crystallographic slip where atoms are being displaced in the direction of the Burger's vector across the slip plane. The deformed configuration is given by the collection of points $x+u(x)$ with $x \in \Lambda$.
The functional $H_\mathrm{AO}$ accounts for the elastic energy which is caused by the displacement $u$ in the presence of the slip field $\sigma$. It should be interpreted as the quadratic approximation of a
more complex, non-linear energy. The model has been introduced in \cite{ao05}, which we refer to for details about its microscopic interpretation. See also Appendix \ref{app:interpretation} for a
heuristic derivation of the model and a discussion about its microscopic meaning.

Note that we study the Ariza-Ortiz model in setting of the FCC lattice because it represents the only simple 3-dimensional lattice, involving only nearest neighbor interactions, satisfying a rigidity estimate a' la Korn, that is, $\sum_{x\sim y}[(u(x)-u(y))\cdot (x-y)]^2\gtrsim \sum_{x\sim y}|u(x)-u(y)|^2$, see eq.~\eqref{c0} below.

\subsection{Symmetries of the Ariza-Ortiz model} \label{sym.sec}
Consider the infinite volume version of the Ariza-Ortiz energy \eqref{AOnoec}, obtained by assuming that $u$ and $\sigma$, rather than satisfying Dirichlet boundary conditions on $\Lambda$, decay sufficiently fast at infinity so that the infinite sum involved in the definition of the energy makes sense.
Such infinite volume Ariza-Ortiz energy is invariant under three different types of symmetry transformations:
\begin{enumerate}
\item Translations: $u \mapsto u + \tau$ where $\tau \in \R^3$ is a constant vector.
\item Linearized rotations: $u \mapsto u + s$ where $s(x) = Sx$ and $S \in \R^{3 \times 3}$ is a
skew symmetric matrix.
\item Gauge invariance: $(u,\sigma) \mapsto (u+v,\sigma+\dd v)$ where $v:\LL \to \LL$ and $\dd v(x,y):= v(y)-v(x)$.
\end{enumerate}
The presence of the first and third symmetry is a direct consequence of the `gradient structure' of the Ariza-Ortiz energy, that is, of the fact that it depends on $u,\sigma$ only upon the combination $\dd u-\sigma$.
Invariance under linearized rotations is an approximation of the invariance under rotations: $u(x) \mapsto R(x+u(x))-x$ for all $R \in SO(3)$.  The invariance of the Ariza-Ortiz energy under linearized rotations is a consequence of the observation that $(\dd u(x,y) + S\,(y-x))\cdot (y-x) = \dd u(x,y)\cdot (y-x)$ for any skew-symmetric matrix $S$. Previously studied models such as the Villain XY model, see, e.g., \cite{FS81, fs82}, are invariant under the analogues of the first and the third symmetries, but in that context there is no analogue of the second symmetry, which is, instead, a distinctive feature of
microscopic models of elasticity.
There are significant consequences resulting from the invariance of linearized rotation, most notably the existence of grains, cf. Theorem~\ref{grain-scaling}.

Note that, in a finite box $\Lambda$ with Dirichlet boundary conditions, the first and second symmetries are broken. On the contrary, the third symmetry is also present in finite volume, provided
that $v$ is chosen to satisfy Dirichlet boundary conditions like $u$. Physically, gauge invariance corresponds to the possibility of conveniently re-labelling the atoms and, correspondingly, of re-defining the nearest neighbours, without any energy cost. Mathematically, gauge invariance implies that the energy only depends on the dislocation part of $\sigma$, defined in the following section.

\section{Main results: Long-range order and grain boundaries}\label{sec.3}
\subsection{Existence of long-range order}
Before defining the Boltzmann-Gibbs distribution we recall from {Sect.~\ref{sym.sec}} the notation $\dd u(x,y)= u(y)-u(x)$ and
that the Ariza-Ortiz energy with Dirichlet boundary conditions on $\Lambda$ is gauge invariant in the sense that
$$ H_{\mathrm{AO}}(u+v,\sigma+\dd v) = H_{\mathrm{AO}}(u,\sigma)$$
for each $v:\LL \to \mathcal L$ that satisfies Dirichlet boundary conditions on $\Lambda$. To remove this degeneracy we say that two slip fields $\sigma$ and $\sigma'$ are equivalent if $\dd\sigma=\dd\sigma'$, with
\begin{align*}
&\dd \sigma:\{(x_1,x_2,x_3) \in \LL^3 \; : \; x_1 \sim x_2\sim x_3\sim x_1\} \to \LL, \\
&\dd \sigma(x_1,x_2,x_3)= \sigma(x_1,x_2)+ \sigma(x_2,x_3)+\sigma(x_3,x_1).
\end{align*}
The function $q=\dd\sigma$ is called the dislocation part of $\sigma$. {Note that, if $\sigma$ satisfies Dirichlet boundary conditions, then also $q$ does, 
 i.e., $q(x_1,x_2,x_3)=0$ for $x_1,x_2,x_3\in\Lambda^c$.}
A discussion of the link between slip fields without dislocations ($\dd \sigma = 0$) and the existence of $v:\LL \to \R^3$ such that $\dd v = \sigma$ can be found in Sect.~\ref{sec.DEC}.

The field $\dd \sigma$ assigns to each triangular face $f$, identified with a $3$-cycle of nearest neighbor sites, a current flowing orthogonally to $f$, in the direction induced by the orientation of $f$.
Typically $q =\dd \sigma$ is decomposed into a sum of {\it dislocation lines}, i.e., $q = \sum_j q_j$, where the supports of the  $q_j$ are the maximal connected components of $\supp q$.
Each of these $q_j$ can be thought of as a current loop. It will be shown in Sect.~\ref{sec.DEC} that $\dd q=0$, where $\dd q$ is the discrete analogue of the curl of $q$: it is a function defined on the elementary cells of $\mathcal L$ that, on each cell, equals the sum of the values of $q$ on the faces of the cell, with the appropriate orientation. In terms of the current loop representation of $q$,
this curl-free condition means that the current loops are closed.

{Denoting by $\mathcal S$ the set of representatives of non-equivalent slip-fields satisfying Dirichlet boundary conditions (i.e., vanishing on edges contained in $\Lambda^c$),}
we are now in a position
to define the expectation of a {gauge-invariant} observable $\varphi$ {(i.e., $\varphi(u,\sigma)=\varphi(u+v,\sigma+\dd v)$ for any $v:\mathcal L\to\mathcal L$ supported in $\Lambda$)} with respect to the Boltzmann-Gibbs distribution by
\begin{equation}
\mathbb{E}_{\beta,\Lambda}(\varphi) =
\frac{1}{Z_{\beta,\Lambda}}\sum_{\sigma \in \mathcal S}\,\int du\, \e^{-\beta\,(H_\mathrm{AO}(u,\sigma) +W(\dd \sigma ))}\, \varphi({u,\sigma}),
\label{Ga_nec}\end{equation}
with
$$ Z_{\beta,\Lambda} = \sum_{\sigma \in \mathcal S}\,\int du\, \e^{-\beta\,(H_\mathrm{AO}(u,\sigma) +W(\dd \sigma ))}$$
and the integral runs over $\mathbb R^{3|\Lambda|}$ (recall that $u(x)\in\mathbb R^3$, for $x\in\mathcal L$, and $u(x)\equiv 0$ if $x\in\Lambda^c$).
The function $W$ represents the energy contribution of the dislocation cores and has the form
\begin{equation}\label{Wq}W(q) = \sum_{f} w(q(f)),\end{equation}
where $w$ is even and the sum runs over the unordered set of faces of $\mathcal L$ (i.e., the set of $3$-cycles of nearest neighbor sites in $\mathcal L$, modulo their orientation).
We assume that $$w(q(f))\ge w_0 |q(f)|^2,$$
for some positive constant $w_0$. We remark that the purely additive structure of the core energy, Eq.\eqref{Wq}, is assumed here just for simplicity: our proofs could be adapted to the case of correlated energies, provided
their correlation decays to zero sufficiently fast at large distances, but we prefer to stick to the assumption of exact additivity here, in order to keep technicalities to a minimum.

{Note that the condition that $\varphi$, like $H_{\mathrm{AO}}$, is invariant under gauge transformations} is a natural requirement: essentially, we are saying that slip fields
differing by exact forms are physically un-distinguishable.

We will be specifically interested in the following observable{s}: for $x,y \in \LL$ and $v_0\in\LL^*$ (the dual of $\mathcal L$, whose basis vectors $m_1,m_2,m_3$ are defined by the conditions
$b_i\cdot m_j=2\pi \delta_{i,j}$, see \eqref{eq:mi}), we let
$${\varphi_{v_0;x}(u)=\cos(u(x)\cdot v_0)}$$
and {we let}
$$\varphi_{v_0;x,y}({u})=\cos((u(y)-u(x))\cdot v_0)$$
{be the corresponding two-point observable.}
{It is apparent that both $\varphi_{v_0;x}(u)$ and $\varphi_{v_0;x,y}(u)$ are gauge invariant, thanks to the condition that $v_0\in\mathcal L^*$.}
The {one-point} observable {$\varphi_{v_0;x}$} is appropriate for {testing the breaking of translational symmetry (i.e., of symmetry 1 in Sect.\ref{sym.sec})
in the presence of Dirichlet boundary conditions: in fact, it is peaked at $u(x)=0$ mod $\mathcal L$ (in particular, it is not invariant under translations $u(x)\to u(x)+\tau$)}, 
and it has zero average under translations of $u(x)$. {Similarly, the corresponding two-point observable $\varphi_{v_0;x,y}$ is appropriate for testing the existence of 
positional long range-order.}

We define the expectation{s}
\begin{equation} \label{def_order.eq}\begin{split}
&{c_{\beta,\Lambda}(v_0;x):=\mathbb E_{\beta,\Lambda} (\cos(u(x)\cdot v_0))},\\
&c_{\beta,\Lambda}(v_0;x,y):=\mathbb E_{\beta,\Lambda} (\cos((u(y)-u(x))\cdot v_0)).\end{split}
\end{equation}
We are interested in taking the thermodynamic limit $\Lambda\to\mathcal L$ that, for boxes $\Lambda=\Lambda^{(N)}$ like in \eqref{LN}, simply indicates the limit $N\to\infty$.
\begin{thm} \label{lro_simp.thm}
{Let $v_0\in\LL^*$.} There are positive constants $C, \beta_0, {r_0},$ which do not depend on $x,y$ and $\beta$
such that, if $\beta > \beta_0$ and {$|x-y|>r_0$}, 
\begin{equation}\begin{split} & {\liminf_{\Lambda\to\mathcal L} c_{\beta,\Lambda}(v_0;x)\ge e^{-C/\beta}},\\
& \liminf_{\Lambda\to\mathcal L} c_{\beta,\Lambda}(v_0;x,y)\ge e^{-C/\beta}.\end{split}\label{eq.thm1}\end{equation}
\end{thm}

Eqs.\eqref{eq.thm1} establish the existence of \textit{{translational symmetry breaking}} and {\it long range {positional} order} in the three-dimensional setting: {
in particular, the first equation implies that, for $\beta$ large enough, the weak limit as $\Lambda\to\mathcal L$ of the 
Gibbs state $\mathbb E_{\beta, \Lambda}$ breaks the translational symmetry $u(x)\to u(x)+\tau$.} {Conversely}, at small enough $\beta$, {it is expected that the limiting Gibbs 
state is invariant under such translational symmetry, that $c_{\beta,\Lambda}(v_0;x)$ decays exponentially to zero in the distance
$\textrm{dist}(x,\Lambda^c)$ as $\Lambda\to\mathcal L$}, and that $\liminf_{\Lambda\to\mathcal L}c_{\beta,\Lambda}(v_0;x,y)$
decays exponentially to zero {as $|x-y|\to\infty$}, because of the screening phenomenon \cite{BM99}, but this remains to be proved for the Ariza-Ortiz model.
The limiting value ${\liminf}_{\Lambda\to\mathcal L} {c_{\beta,\Lambda}(v_0;x)}$, which passes from being positive at large $\beta$ to being (at least conjecturally)
identically zero at low $\beta$, has the interpretation of
order parameter for positional order.

The reason why we write $\liminf_{\Lambda\to\mathcal L}$ rather than $\lim_{\Lambda\to\mathcal L}$ in \eqref{eq.thm1} is that
a priori we do not know whether the limit exists: our system is of Coulomb-type and the standard theory of the existence of the thermodynamic limit does not apply directly.
There are several results in the literature about the existence of the thermodynamic limit of Coulomb systems in 3D, but they do not apply literally to our case,
see e.g. \cite{FP78, LL72} and the review \cite{BM99} and references therein. It is likely that they could be adapted to our context as well, but this is beyond the scope of our paper.

Theorem \ref{lro_simp.thm} is a consequence of Theorem~\ref{lro.thm} which is stated with proof in Section~\ref{proof_lro.sec}. Theorem~\ref{lro.thm} and its proof
provide a more detailed estimate than \eqref{eq.thm1}: in particular, they show that
{both $c_{\beta,\Lambda}(v_0;x)$ and $c_{\beta,\Lambda}(v_0;x,y)$} factor exactly into the product of two contributions, one associated with a Gaussian average (the `spin wave contribution') and one associated with an effective theory for the dislocation cores.
The first contribution is explicit, and asymptotically equal, {as $\Lambda\to\mathcal L$ and $|x-y|\to\infty$, to $e^{-C_0/(2\beta)}$ resp. $e^{-C_0/\beta}$ 
for the one-point resp. two-point observable}, for an explicit constant $C_0$. The second is bounded via cluster expansion and the use of Jensen's inequality, following the same strategy of \cite{fs82,KK86}, and leads to an exponentially small correction to 
{the spin-wave contribution.} Note that the assumption $d=3$ plays a key role both in the computation of the spin wave contribution 
(Section \ref{sec.swme}) and in the estimate of the correction due to the dislocation cores (Section~\ref{sec.disme}).

\medskip

Theorem \ref{lro_simp.thm} is analogous to \cite[Theorem 3]{BCHMR18} that, however, refers to long-range {orientational} order in a {\it mesoscopic} model for a solid with dislocations.
An important difference between our setting and the one in \cite{BCHMR18} concerns the modeling part. While our model,
even though simplified, has a direct microscopic interpretation, theirs involves an auxiliary set of currents, whose microscopic interpretation is not immediate. It is likely that the model in \cite{BCHMR18}
could be obtained starting from a more fundamental atomistic one, via a suitable coarse graining procedure. It would be very interesting to substantiate this expectation by rigorous results.
From a technical point of view, the tensorial structure of the Ariza-Ortiz Hamiltonian on the FCC lattice introduces some extra difficulties, compared to \cite{BCHMR18}, in the reduction
to an effective model of dislocations and in the treatment thereof, which we solve thanks to the tools of discrete exterior calculus, reviewed below. On the other hand, the general strategy of our proof is
analogous to that in \cite{BCHMR18}, in that both rely on the ideas of \cite{fs82,KK86}.

\medskip

If the dimension is one or two then the existence of long-range positional order is prevented by the Mermin-Wagner theorem \cite{M68, MW66}, see also \cite{FP81,ISV,Pf81,Rich}. However,
the Mermin-Wagner theorem does not prevent the possibility of having orientational order in two dimensions; actually, spin wave theory suggest that orientational order should be present in two dimensions
\cite{M68}. It is not a priori clear, not even heuristically or intuitively, whether the presence of dislocations, and in particular of grains, can destroy the prediction based on spin wave theory. Therefore, it would be extremely interesting to prove or disprove the existence of long-range orientational order in a concrete atomistic model for a two-dimensional elastic crystal with dislocation. Probably, the simplest such model is the analogue of the model studied in this paper,
in a two-dimensional setting (e.g., in the case that the 3D FCC lattice is replaced by the 2D triangular lattice). We expect that the methods developed in \cite{FS81} for the study of the Kosterlitz-Thouless transition in the 2D Villain rotator model may be adapted to such a case. We plan to come back to this problem in a future publication.

\subsection{Energy scaling of grains}
An important question which cuts to the core of the crystal problem is whether the Ariza-Ortiz model accounts for structures such as grain-boundaries. Grains only exist because of the additional symmetry of atomistic systems: (linearized) rotational symmetry, sometimes also referred to as `objectivity'. More precisely, let $S \in \R^{3 \times 3}$ be a skew-symmetric matrix and $\mathcal G \subset \LL$ be the location of the grain, which we assume to be simply connected and bounded.
We say that a pair $(u,\sigma)$, with $u:\mathcal L\to \mathbb R^3$ a displacement field and $\sigma: E_1\to {\mathcal L}$ a lattice-valued slip field (here $E_1$ is the set of nearest neighbor pairs of $\mathcal L$ and $\sigma(x,y)$ is assumed to be odd under orientation flip $(x,y)\to (y,x)$)
supports a `perfect grain' $\mathcal G$ with orientation $S$ if it is {\it gauge equivalent} to a configuration $(u',\sigma')$ such that
\begin{eqnarray} \label{gsupp}
&& u'(x)-u'(y) = \begin{cases} S (x-y) & \text{ if } \{x,y\} \subset \mathcal G,\\
0 & \text{ if } \{x,y\} \subset \mathcal G^c,
\end{cases}\\
&& \label{gsupp2}\sigma'(x,y)=0\quad \text{if}\quad x\sim y \quad \text{and}\quad \{x,y\} \subset \mathcal G\quad \text{or}\quad \{x,y\} \subset \mathcal G^c.\end{eqnarray}
We say that $(u',\sigma')$ is gauge equivalent to $(u,\sigma)$, if $(u',\sigma')=(u+v,\sigma+\dd v)$ for some lattice valued function $v$, see Sect.\ref{sym.sec}.
Note that \eqref{gsupp}-\eqref{gsupp2} do not impose any constraint on the nearest neighbour
bonds $(x,y)$ such that $x\in\mathcal G$ and $y\in\mathcal G^c$, or viceversa. For later reference, we denote this set of bonds by $E_1^{\mathrm{b}}(\mathcal G)$ (`b' for `boundary'):
$$E_1^{\mathrm{b}}(\mathcal G) = \{(x,x') \in E_1 \; : \; \{x,x'\} \cap \mathcal G \neq \emptyset \text{ and }
\{x,x'\} \cap \mathcal G^c \neq \emptyset\},$$
and we let $|E_1^{\mathrm{b}}(\mathcal G)|$ denote the number of elements of $E_1^{\mathrm{b}}(\mathcal G)$ modulo orientation.
It is not obvious from the outset whether a pair $(u,\sigma)$ supporting a perfect grain $\mathcal G$ with orientation $S$ can be chosen such that the associated energy is smaller than the volume of $\mathcal G$. For example, the pair $(u_S,0)$, with
\begin{equation} \label{defus.eq}
 u_S(x) = \begin{cases} Sx+\tau & \text{ if } x \in \mathcal G,\\
0 & \text{ else}\end{cases}
\end{equation}
clearly supports a perfect grain $\mathcal G$ with orientation $S$, for any fixed $\tau \in \R^3$. However, it can be
easily checked that $\min_\tau H_\mathrm{AO}(u_S, 0) \sim |E_1^{\mathrm{b}}(\mathcal G)| \,\mathrm{diam}(\mathcal G)^2 \gtrsim |\mathcal G|^\frac{4}{3}$ if $|\mathcal G|$ is large. The fact that the energy
$H_\mathrm{AO}(u_S, 0)$ is larger than $|\mathcal G|$ is a consequence of the discontinuity across the boundary of $\mathcal G$. However, one should not conclude from this that the
`optimal grain energy' scales more than extensively: on the contrary, it is remarkable that, for large grains, it scales proportionally to the size of the grain boundary $E_1^{\mathrm{b}}(\mathcal G)$,
as summarized in Theorem \ref{grain-scaling} below. Of course, before formulating our result, we first need to clarify what we mean by `optimal grain energy'.
Let us remark that there is no unique, well-established, notion of
grain energy at a microscopic level. In fact, such a notion is well-defined at the mesoscopic level, in which case it is related to the distribution of dislocations at the boundary of the grain, and is
known to scale proportionally to the grain boundary, see e.g. \cite{SB06}. Here we propose a microscopic definition thereof, which we expect to reduce to the usual mesoscopic definition in an appropriate
scaling limit. A rigorous connection of our microscopic definition with the continuum one is an interesting open problem, which goes beyond the purposes of this paper. 

We let the optimal grain energy be defined as
\begin{equation}\label{eq:GSen}\mathcal E_{\mathcal G}(S):=\liminf_{\Lambda\to\mathcal L}\lim_{\epsilon\to 0^+}\inf_{\sigma\in \mathcal M_S^{(\epsilon)}(\mathcal G)}\inf_u H_{\text{AO}}(u,\sigma),\end{equation}
where $\mathcal M_S^{(\epsilon)}(\mathcal G)$ is the set of $\epsilon$-minimizers of $H_{AO}(u_S,\sigma)$, with $u_S$ as in \eqref{defus.eq}, over the slip fields compatible with the grain $\mathcal G$:
more explicitly, $\mathcal M_S^{(\epsilon)}(\mathcal G)$ is the set of lattice-valued slip fields $\sigma$ such that there exists $\tau$ (see \eqref{defus.eq}) for which $\sigma,\tau$ realize the infimum of $\inf_{\tau}\inf_\sigma^*
H_{\text{AO}}(u_S,\sigma)$  within a precision $\epsilon$, where the $*$ on $\inf^*_\sigma$ indicates the constraint ${\rm supp}\,\sigma \subseteq E_1^b(\mathcal G)$. We are ready to state the basic bound
on $\mathcal E_{\mathcal G}(S)$, showing, as anticipated above, that the optimal grain energy scales like the grain boundary for large grains.

\begin{thm} \label{grain-scaling}
For any skew symmetric matrix $S \in \R^{3\times 3}$ and any bounded connected set $\mathcal G$,
\begin{align} \label{surf-scal}
\mathcal E_{\mathcal G}(S)
\leq 6\, |E_1^{\mathrm{b}}(\mathcal G)|.
\end{align}
\end{thm}
It is very likely that our upper bound can be improved, i.e., it is not sharp.
In fact, the construction of matching upper and lower bounds in the limit of a large grain constitutes an interesting mathematical problem.

\begin{proof}
In order to prove the theorem, we will construct a lattice valued slip field $\sigma_S$ supported on $E_1^{\mathrm{b}}(\mathcal G)$ such that $H_{\mathrm{AO}}(u_S,\sigma_S)\le 6|E_1^{\mathrm{b}}(\mathcal G)|$, where $u_S$ is defined as in \eqref{defus.eq}, with $\tau=0$. We first determine slip amplitudes $\xi_{(l,n)} \in \R$ such that the matrix $S$ can be decomposed into simple slip systems, i.e.
\begin{equation} \label{slip_decomp} S = \sum_{(l,n)} \xi_{(l,n)}\, b_l \otimes m_{n}
\end{equation}
with the convention that $b_l \in \LL$ are the slip vectors and $m_n \in \mathcal L^*$ are the slip plane normals. The standard 12 slip systems of the FCC lattice are
$$(l,n) \in \{(1,2),(1,3),(2,1),(2,3),(3,1),(3,2),(4,1),(4,4),(5,2),(5,4),(6,3),(6,4)\},$$
where: $b_1,\ldots,b_6$ are the nearest neighbor vectors of the FCC lattice, introduced at the beginning of Sect.\ref{secAO},
$m_1,m_2,m_3$ are the basis vectors of $\mathcal L^*$, see \eqref{eq:mi}, and $m_4=m_1+m_2+m_3$.
A simple calculation delivers the following solution of~\eqref{slip_decomp}
\begin{eqnarray}\label{rotdec}
S &=&
\frac{1}{4\pi}\Bigl(S_{12}\,[b_3\otimes (m_1 -m_2) + b_6\otimes (m_3+m_4)]
+S_{13}\,[(b_2\otimes (m_1-m_3)\nonumber\\
&&-b_5\otimes (m_2+m_4)]+S_{23}\,[b_1\otimes(m_2-m_3) +b_4\otimes (m_1+ m_4)]\Bigr).\nonumber
\end{eqnarray}

Once the slip amplitudes $\xi_{(l,n)}$ are fixed, we let
$$\sigma_S(x,y)=-\sigma_S(y,x):=-\sum_{(l,n)}b_l\,\lfloor \xi_{(l,n)}\,x\cdot m_n\rfloor, \quad \text{if}\quad x\sim y,\quad \text{with}\quad x\in\mathcal G\ \text{and}\ y\not\in\mathcal G,$$
and $\sigma_S(x,y)=0$ otherwise.

Let us now compute $H_{\mathrm{AO}}(u_S,\sigma_S)$. We partition the set $E_1$ into three groups:
$$
E_1 = E_1^\mathrm{i}(\mathcal G) \cup E_1^\mathrm{o}(\mathcal G) \cup E_1^\mathrm{b}(\mathcal G),$$
where $E_1^\mathrm{i}(\mathcal G)$ is the set of bonds inside $\mathcal G$, while $E_1^\mathrm{o}(\mathcal G)$ is the set of bonds outside $\mathcal G$.
The partition of $E_1$ induces a decomposition of the energy:
$$ H_\mathrm{AO}(u_S,\sigma_S) = H^\mathrm{inside}_\mathrm{AO}(u_S,0)+H^\mathrm{outside}_\mathrm{AO}(u_S,0)+H^\mathrm{boundary}_\mathrm{AO}(u_S,\sigma_S),$$
where we used the fact that $\sigma_S$ is zero on $E_1^\mathrm{i} \cup E_1^\mathrm{o}$. Now, recalling that $\dd u_S(x,y)=S(x-y)$ for $(x,y)\in E_1^{\mathrm{i}}$, we find
$H^\mathrm{inside}_\mathrm{AO}(u_S,0)=0$, by the invariance under linearized rotations. Moreover, $H^\mathrm{outside}_\mathrm{AO}(u_S,0)=0$, simply because $u_S(x)=0$ for $x\in\mathcal G^c$.
Finally, by the very definition of $u_S$ and $\sigma_S$,
\begin{equation} \label{nobulk}
H_\mathrm{AO}^\mathrm{boundary}(u_S,\sigma_S)=\frac12\sum_{\substack{(x,y)\in E_1^b(\mathcal G): \\ x\in\mathcal G,\, y\in\mathcal G^c}}\sum_{(l,n)}\big[(x-y)\cdot b_l\, \big(\xi_{(l,n)}\,x\cdot m_n-\lfloor
\xi_{(l,n)}\,x\cdot m_n\rfloor\big)\big]^2.
\end{equation}
Now, the difference in parentheses in the right side is between $0$ and $1$. Therefore, recalling that $|x-y|=|b_l|=1$ and that the sum over $(l,n)$ runs over 12 different terms, we find
$$H_\mathrm{AO}^\mathrm{boundary} \leq  6 |E_1^b(\mathcal G)|,$$
as desired. \end{proof}
In order to visualize the `optimal' location of the atoms within a grain, we remark that the pair $(u_S,\sigma_S)$ used in the proof of Theorem \ref{grain-scaling} is gauge equivalent to a configuration $(u,\sigma)$ such that: (1)
$|u(x)|\le 6$ for $x\in \mathcal G$, and $u(x)=0$ otherwise, (2) the support of $\sigma$ is contained in $E_1^\mathrm{i}(\mathcal G)$. In order to exhibit such an equivalent pair, we let $u=u_S+v_S$ and $\sigma=\sigma_S+\dd v_S$, with
$(u_S,\sigma_S)$ the same as those used in the proof of the theorem, and
$$v_S(x)=\begin{cases} -\sum_{(l,n)}b_l\lfloor \xi_{(l,n)}x\cdot m_n\rfloor, \quad \text{if} \quad x\in\mathcal G,\\
0, \quad \text{otherwise}.\end{cases}$$
We visualise in left panel of Fig.~\ref{grains} such a displacement field $u$ in a two-dimensional setting where
$S = \frac{1}{5}\left({0 \atop 1}{-1 \atop 0}\right)$ and Neumann boundary conditions are used (see Sect.\ref{subdec} for a definition of Neumann boundary conditions). The colored triangles are the support of $\dd \sigma$.
The minimizer of $H_\mathrm{AO}(\cdot,\sigma)$ is shown in the right panel of Fig.~\ref{grains}. The corresponding minimal energy, $\inf_u H_\mathrm{AO}(u,\sigma)$,
is the one that the system will reach after relaxation at fixed slip field $\sigma$. In the limit of large grain, it is supposed to provide a good approximation for the optimal grain energy $E_{\mathcal G}(S)$ in
\eqref{eq:GSen}. As it will be proved in the following sections, remarkably, the minimal energy $\inf_u H_\mathrm{AO}(u,\sigma)$ only depends on the `charge distribution' $q=\dd \sigma$, which, therefore, characterizes the grain from an energetic point of view.

\begin{figure}
\centering
\begin{subfigure}{.5\textwidth}
\centering
\includegraphics[width=.9\linewidth]{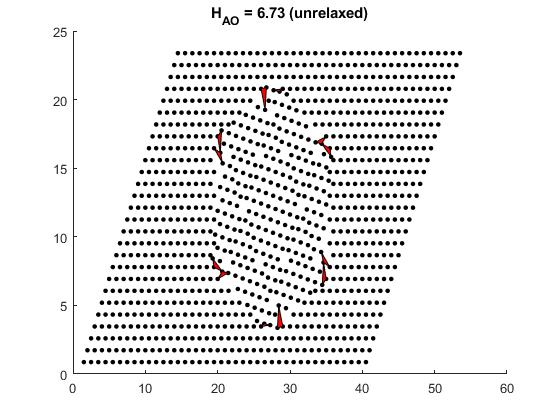}
\end{subfigure}%
\begin{subfigure}{.5\textwidth}
\centering
\includegraphics[width=.9\linewidth]{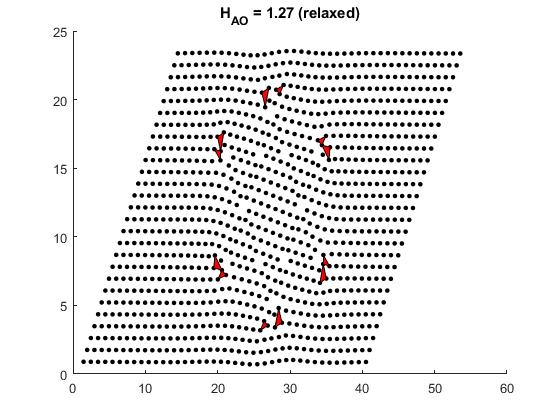}
\end{subfigure}
\caption{The left panel shows the displacement $u=u_S+v_S$ for the 2-dimensional Ariza-Ortiz model on the triangular lattice with Neumann boundary conditions and
$S =\frac15 \left( {0 \atop 1} {-1 \atop  0}\right)$.
Colored triangles indicate the support of $\dd \sigma$, with $\sigma=\sigma_S+\dd v_S$.
The right panel shows the relaxed displacement field $u_\sigma$ which minimizes $H_\mathrm{AO}(\cdot, \sigma)$.}
\label{grains}
\end{figure}

\subsection{Read-Shockley law}\label{subsec:RS}
Theorem \ref{grain-scaling} shows that the optimal energy of a perfect grain scales like its boundary, but does not provide an explicit formula for the surface tension, that is, the proportionality constant in front of $|E_1^{\mathrm{b}}(\mathcal G)|$,
in the limit of a large grain. Physically, there are explicit expectations for the surface tension, specifically in the limit of small rotation angles:
according to the Read-Shockley formula \cite{RS50}, given a large grain, rotated by a small angle $\theta$ with respect to a reference crystalline background, its total energy is proportional to its boundary,
with a proportionality constant $\gamma(\theta)$ of the form \eqref{readsh}.
An upper bound which is consistent with the logarithmic scaling can be found in \cite{LL17}.

In this section, we state two results about the exact, asymptotic, computation of the energy of a dislocation dipole and of two walls of dislocations with opposite charges, far away from each other.
In particular, the energy of the two parallel walls of dislocations with opposite charges is expected to correspond to the optimal energy of a grain supported in the region between the two walls (the electrostatic analogue to keep in mind is a capacitor:
dislocations correspond to the charges on the plates of the capacitor, and the intermediate region between the plates is where the elastic energy concentrates), in the sense of definition \eqref{eq:GSen}.
The reader can convince herself/himself that the smaller the density of dislocations on the walls, the smaller the rotation angle of the grain, and that in the limit of small density of dislocations,
the rotation angle goes to zero linearly with the density. Therefore, the computation of the energy of the `dislocation capacitor' performed below provides information on the optimal energy of the corresponding grain.
Our main result is that {\it we recover the Read-Shockley law} for a grain with such a simple, specific, geometrical shape.

The computations are reported in Sect.\ref{sec.RS}. For simplicity, we perform the computations in two dimensions, but similar results, including the logarithmic dependence of the surface tension on
the rotation angle, in the sense of \eqref{readsh}, can be extended to
three dimensions, by assuming that the distribution of dislocations under consideration is translationally invariant in the third coordinate direction; however, in three dimensions the computations become cumbersome and
their key features would be hidden behind unimportant technical complications: therefore, we prefer to restrict to 2D and leave the tedious but straightforward extension to higher dimensions to the interested reader.

We denote by $\mathcal T$
the triangular lattice and, with some abuse of notation, we let its basis vectors be $b_1=\left({1\atop 0}\right)$, $b_2=\frac{1}{2}\left({-1\atop \sqrt3}\right)$. For later reference we also define $b_3=-b_1-b_2=\frac{1}{2}\left({-1\atop -\sqrt3}\right)$.
Given a finite box $\Lambda\subset \mathcal T$ of side $N$ (the 2D analogue of \eqref{LN}), we let the 2D Ariza-Ortiz energy in $\Lambda$ with Dirichlet boundary conditions be defined by the same formula \eqref{AOnoec}; with
some abuse of notation, we denote the 2D energy by the same symbol $H_{\mathrm{AO}}(u,\sigma)$.

\begin{figure}
\begin{center}\includegraphics[width=.7\textwidth,page=4]{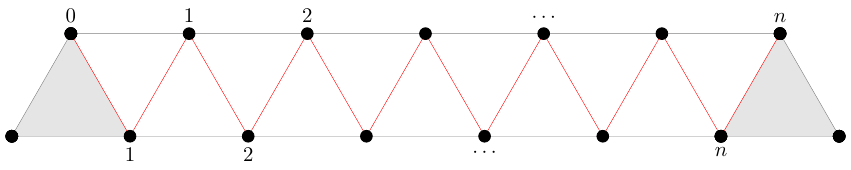}\end{center}
\caption{Graphical illustration of the charge distribution $q_\mathrm{dip}^n$, for $n=6$. The shaded triangles, corresponding to faces $f_0$ and $f_n$, indicate the support of $q_\mathrm{dip}^n$.
We also show in red the support of a slip field $\sigma^n_\mathrm{dip}$ such that $\dd \sigma^n_{\mathrm{dip}}=q^n_\mathrm{dip}$, see \eqref{defsigdip.eq}.
The sites labelled $j=1,\ldots,n$ on the bottom (resp. top) row have coordinates $jb_1$ (resp. $jb_1-b_3$).}\label{fig.app}
\end{figure}

We consider a {\it dislocation dipole} formed by a pair of opposite charges $\pm b_1$, separated by a distance $n$ in direction $b_1$, whose `charge distribution' is:
\begin{eqnarray}  \label{disl_dist}
q_\mathrm{dip}^n &=& \left({\bf 1}_{f_0} - {\bf 1}_{f_n}\right) b_1
\end{eqnarray}
with $f_n=(0,b_1,-b_3)+nb_1$, see Fig.~\ref{fig.app}.
We also consider two parallel arrays of dislocations, formed by $M$ dislocation dipoles as in \eqref{disl_dist}, arranged one at a distance $m \sqrt{3}$ from the other in the direction orthogonal
to $b_1$, whose charge distribution is:
\begin{equation} \label{wall_dist}
q_\mathrm{grain}^{M,n,m}(f) = \sum_{j=1}^M q^n_\mathrm{dip}(f-jm(b_2-b_3)).
\end{equation}
In the limit $M\to\infty$, the charge distribution $q_\mathrm{grain}^{M,n,m}$ tends to that of two infinite walls of dislocations, separated by a distance $n$, with charge density $\sim 1/m$. As discussed
in Appendix \ref{app.grain}, its energy is expected to coincide at dominant order with the optimal energy of a grain supported in the region between the walls, rotated by an angle $\theta\sim 1/m$, in the limit $m\to\infty$.

\begin{thm} \label{thm.RS}
Let
\begin{eqnarray}  E_\mathrm{dip}(n)&=&\lim_{\Lambda\to\mathcal T}\min \left\{H_{AO}(u,\sigma)\; : \; \dd \sigma = q^{n}_\mathrm{dip}\right\},\\
E_\mathrm{grain}(n,m)&=&\lim_{M\to\infty} \frac{1}{\sqrt3 mM} \Big[\lim_{\Lambda\to\mathcal T} \min \left\{H_\mathrm{AO}(u,\sigma) \; : \; \dd \sigma = q_\mathrm{grain}^{M,n,m}\right\}\Big]
\label{defEgrain}\end{eqnarray}
be the energy of a dipole and the energy density of a grain boundary per unit length, in the thermodynamic limit.
Then
\begin{equation}\label{dip.0}E_\mathrm{dip}(n)=\frac{\log n}{2\pi\sqrt3}+O(1), \quad n\gg 1,\end{equation}
and
\begin{equation}\lim_{n\to\infty}E_\mathrm{grain}(n,m)=\frac{\log m}{6\pi m} + O(1/m), \quad m\gg 1.\label{RS.0}\end{equation}
\end{thm}

The proof of Theorem \ref{thm.RS} is given in Sect.\ref{sec.RS}.
Eq.\eqref{RS.0} is the desired Read-Shockley law for the energy of a grain boundary.
Its remarkable feature is that it is asymptotically independent of the separation among the two arrays of charges it consists of. This is in sharp contrast with the
`capacitor law', i.e., with the formula for the energy of two parallel arrays of `scalar' dipoles, i.e., of a similar arrangement of charges in the usual Coulomb lattice gas, which scales linearly
in $n$ at large separation $n$. For a technical comparison of the computations leading to the Read-Shockley and the capacitor laws, see Sect.\ref{sec.compar} below.

Note that the $E_{\mathrm{grain}}(n,m)$ does not include a contribution from the dislocation cores.
Of course, the inclusion of such a contribution, of the form $W(q)$, see \eqref{Wq}, can be done without any additional difficulty. Note that
the extra energy from the dislocation cores would contribute $O(1/m)$ to the right side of \eqref{RS.0} and, therefore, would not modify the dominant asymptotics of the Read-Shockley law.

\section{Exterior calculus}\label{sec.DEC}
In this section, we review a few basic aspects of discrete exterior calculus, which is a fundamental tool used in the proof of the main results. In particular, an application of the Hodge decomposition to the Ariza-Ortiz model will allow
us to decompose its energy in the sum of a `spin wave' part plus a `dislocation' part: such a decomposition is central to our analysis and will be used systematically in the following.

\subsection{Cellular complex, discrete $p$-forms and discrete differential}\label{subdec}

The domain of the three-dimensional Ariza-Ortiz model is given by cells consisting of
\begin{itemize}
\item vertices $E_0$,
\item oriented edges $E_1$ (ordered vertex pairs),
\item oriented faces $E_2$ (polygons whose sides are consistently oriented edges)
\item oriented volumes $E_3$ (polyhedra whose faces are consistently oriented faces),
\end{itemize}
which form a {\it cellular complex}, cf \cite{Hat02}. The orientation of a face $f\in E_2$ is defined by the direction of a reference vector, orthogonal to $f$; the sides of $f$ are said to be consistently oriented if their orientation satisfies the `right-hand rule'.
The orientation of a volume $v\in E_3$ is either `outward' or `inward': its faces are said to be consistently oriented if the directions of their reference vectors all point, correspondingly, in the outward or inward
direction. The case that is of interest to us is where the vertices coincide with a Bravais lattice, a case that is commonly referred to as {\it lattice cellular complex}.
We are specifically interested in the case that $E_0=\mathcal L$, with $\LL$ the face centered cubic lattice, in which case we let, in particular, $E_1$ be the set of all ordered pairs of `nearest neighbor' sites (those at smallest Euclidean distance), and $E_2$ the set of oriented triangular faces associated with the $3$-cycles of nearest neighbor sites. A detailed description of the corresponding cellular complex is given in Appendix~\ref{app.FCC}.

The {\it boundary operator} $\partial_p:E_p\to\ E_{p-1}$ with $p>0$ returns the set of boundary cells with the appropriate orientation.
By repeated applications of the boundary operator, any $p$-cell $c$ with $p>0$ is mapped to a set of vertices in $E_0$, which we refer to as the `set of vertices of $c$' and denote by $V(c)$.
We only require a small subset of
cohomology theory and will use a minimalistic setup. In particular, the action of $\partial_p$ is defined via explicit formulae in Appendix~\ref{app.FCC}, the reader is encouraged to confirm that it coincides with the standard definition \cite[Sec 3]{Hat02}.

\medskip

The vector space $\form^p$ is the set of $p$-forms, namely the set of functions $u:E_p\to \mathbb R^3$ that are odd under orientation flip. The {\it lattice-valued}
$p$-forms, that is, those that return values in $\mathcal L$, will be denoted by $\form^p_{\mathcal L}$.
We define for $p=0,1,2$ the {\it exterior derivative
operators} $\dd_p:\form^p \to \form^{p+1}$. If $p=1,2$, they are given
by the formula
\begin{equation}\label{dp} \dd_p u(c) = \sum_{c' \in \partial_{p+1} c}u(c'),\qquad c\in E_{p+1},\end{equation}
if $p=0$ and $e=(x,y)\in E_0$ is an oriented edge then
\begin{equation} \label{dp1} \dd_0 u(e)=u(y)-u(x).\end{equation}
A straightforward calculation shows that $\dd_{p+1} \dd_p =0$, for $p=0,1$, see, e.g., \cite[Lemma 2.1]{Hat02}. In some cases, it is useful to interpret $\dd_{p+1} \dd_p$ as being $=0$ also for $p=2$, in
which case we let $\dd_3:=0$. Whenever the notation is un-ambiguous, we will drop the label $p$ from $\dd_p$ (i.e., if it is clear from the context that $u$ is a $p$-form, then we will write
$\dd u$ instead of $\dd_p u$).

\medskip
We are interested in the cellular complexes and the corresponding set of $p$-forms,
obtained by taking {\it finite} portions $\Lambda$ of $\mathcal L$, with prescribed boundary conditions, namely Dirichlet, Neumann, or periodic.
For simplicity, we restrict to cases in which such finite portions are parallelepipeds of size $N$, like in \eqref{LN}.

In the case of Neumann boundary conditions, we let $\Lambda_p$ be the subset of $E_p$ consisting of the
$p$-cells $c$, whose set of vertices are contained in $\Lambda$; the $p$-forms of interest are those that depend only on the $p$-cells in $\Lambda_p$.
In the case of Dirichlet boundary conditions or {periodic} boundary conditions we maintain the same cellular complex as for $\LL$.
For Dirichlet boundary conditions
the relevant $p$-forms are those that assume non-zero values on cells whose vertices have non-empty intersection with $\Lambda$.
For periodic boundary conditions the $p$-forms of interest are $N$-periodic in the directions $b_1,b_2,b_3$.
In all these cases, with some abuse of notation, we denote the cellular complex by $(\Lambda_0,\Lambda_1,\Lambda_2,\Lambda_3)$
and by $\form^p$ the corresponding sets of $p$-forms.

\medskip

For any given finite $\Lambda$ as in \eqref{LN} and all the three boundary conditions introduced above, the vector spaces $\form^{p}$ are finite dimensional Hilbert spaces with canonical inner product\footnote{{We use the convention that, for $u,v\in\form^{0}$, $\langle u,v\rangle=\sum_{x\in\Lambda_0} u(x)\cdot v(x)$, while, for $u,v\in\form^{p}$ with $p=1,2,3$, 
$\langle u,v\rangle=\frac12\sum_{c\in\Lambda_p} u(c)\cdot v(c)$, where we recall that, for $p>0$, $\Lambda_p$ is the set of {\it oriented} $p$-cells: therefore, the factor $1/2$ in front of the sum is chosen so that every unoriented $p$-cell is effectively counted just once.}} $\langle\,\cdot\,,\,\cdot\,\rangle$.
Thanks to the relation $\dd_{p} \dd_{p-1}=0$, for $p=1,2,3$, one has that $\range \dd_{p-1} \subset \nullo \dd_p$ and we can define the cohomology groups
$$ H^p = \nullo \dd_p / \range \dd_{p-1},\qquad p=0,1,2,3,$$
with the conventions that: $/$ denotes the standard quotient operator, $\nullo \dd_3=\form^3$, and $H^0=\nullo \dd_0$.
As usual, we say that:
\begin{itemize}
\item if $u \in \form^p$ has the property that $\dd u =0$ (i.e., if $u\in\nullo \dd_p$), then $u$ is closed;
\item if $u \in \form^p$ has the property that $u = \dd v$ for some $v \in \form^{p-1}$ (i.e., if $u\in\range\dd_{p-1}$), then $u$ is exact.
\end{itemize}
In terms of these definitions, $H^p$ is the subspace of closed $p$-forms modulo the exact $p$-forms (i.e., modulo the following equivalence relation for closed $p$-forms: $u_1\sim u_2$ $\Leftrightarrow$ $u_2-u_1=\dd v$, for some $v\in \form^{p-1}$). The space $H^p$ characterizes the obstructions to the
solvability of the equation $\dd u = v$ if $v \in \form^{p}$ is closed. If $H^p=\{0\}$, then
any closed $v$ is automatically exact. More generally, $v \in \form^p$ is exact if and only if it is closed and additionally satisfies $\dim H^p$ linear constraints.

The cohomology groups associated with the box $\Lambda$ with Dirichlet, Neumann and periodic boundary conditions are known, and are the following.
\begin{description}
\item[Dirichlet boundary conditions]
\begin{equation}\label{coh.Dir} H^p = \begin{cases} \R & \text{ if } p = 3,\\
\{0\} & \text{ else.}
\end{cases}\end{equation}
\item[Neumann boundary conditions]
\begin{equation}\label{coh.Neu}  H^p = \begin{cases} \R  & \text{ if } p = 0,\\
\{0\} & \text{ else.}
\end{cases}\end{equation}
\item[Periodic boundary conditions]
\begin{equation}\label{coh.per}  H^p= \begin{cases} \R & \text{ if $p=0$ or $p=3$},\\
\R^3 & \text{ if $p=1$ or $p=2$}.
\end{cases}\end{equation}
\end{description}
In order to prove these formulas,
note that the box $\Lambda$ is topologically equivalent to a $3$-dimensional ball $B^3\subset \mathbb R^3$.
Therefore, the cohomology for Neumann boundary conditions corresponds to the de Rham cohomology of the ball, which is
given by the Poincar\'e Lemma \cite{BoTu82}.
The cohomology with Dirichlet boundary conditions corresponds to the cohomology with {\it compact support} for the ball $B^3\subset \mathbb R^3$; the result then follows from Poincar\'e duality between de Rham cohomology and
cohomology with compact support \cite{BoTu82}. Finally, the cohomology with periodic boundary conditions is the de Rham cohomology of a
three-dimensional torus $T^3\sim S^1\times S^1\times S^1$; the result is then an application of K\"unneth formula \cite{BoTu82}.

In the following, we will also need a quantitative version of the Poincar\'{e} Lemma for lattice valued $2$-forms, in the form stated next.

\begin{prop}\label{lat-Poincare} Consider the cellular complex associated with a finite portion $\Lambda$ of the FCC lattice $\mathcal L$, as in \eqref{LN}, with Dirichlet or Neumann boundary conditions, together with the
associated set of $p$-forms (note, in particular, that $H^2=\{0\}$). Let $q \in \form^2_\LL$ be closed, with finite support. Let $B \subset \R^3$ be the smallest
parallelepiped with edges parallel to the basis vectors $b_1$, $b_2$, $b_3$ such that $\supp q\subseteq B$.
There exists a constant $c>0$, independent of $\Lambda$, and a 1-form $n \in \form^{1}$ with the following properties:
\begin{enumerate}
\item $n$ is $\LL$-valued,
\item $\dd n = q$,
\item the support of $n$ is contained in $B$,
\item $\max_{e \in \Lambda_1} |n(e)| \leq c \langle q, q \rangle^2$.
\end{enumerate}
\end{prop}
The proof is a straight-forward adaptation of Lemma~{3.2} in \cite{KK86}.

\subsection{Hodge decomposition}\label{HD}
In this section, we obtain a representation of $H_\mathrm{AO}(u,\sigma)$ in terms of $\dd \sigma$. Setting $q = \dd \sigma$ we can choose a representative displacement- and slip-field $(u_\sigma,\sigma_q) \in \form^0\times \form^1$ such that $\dd \sigma_q = q$ and $\dd u_\sigma = \sigma-\sigma_q$. With such a choice
$$ H_\mathrm{AO}(u,\sigma) = H_\mathrm{AO}(u-u_\sigma,\sigma_q).$$
Interestingly it is possible to choose $(u_\sigma,\sigma_q)$ so that the energy decomposes into
a purely elastic part and a dislocation part
$$ H_\mathrm{AO}(u,\sigma) = H_\mathrm{AO}(u-u_\sigma,0) + H_\mathrm{AO}(0,\sigma_q), $$
cf Theorem~\ref{energy-decomp} below. In general $\sigma_q$ is not $\mathcal L$-valued, which means that
a physical interpretation is not obvious. While the additive decomposition simplifies our analysis
significantly, we believe that it is not central for the validity of our main results.

To establish the additive decomposition of the Ariza-Ortiz energy we employ the classic
Hodge decomposition. The fundamental idea is to construct the relevant $u_q$ and $\sigma_q$ in terms of solutions of the Poisson equation.

The Laplace operator $\Delta_p:\form^p \to \form^p$ is defined by
\begin{equation}\label{Lap}\Delta_p:=\dd_{p-1} \dd^*_{p-1}+\dd^*_{p}\dd_p,\qquad p=0,1,2,3\end{equation}
where $\dd^*_{p-1}:\form^p\to\form^{p-1}$ is the adjoint of $\dd_{p-1}$, with respect to $\langle \cdot,\cdot\rangle$ (for $p=0$, \eqref{Lap} should be interpreted as $\Delta_0=\dd^*_0\dd_0$, i.e.,
$\dd_{-1}=\dd^*_{-1}:=0$). Also in this case, as in Sect.\ref{sec.DEC}, we will drop the dependence upon $p$ whenever the space, which $\Delta_p$ or $\dd^*_{p-1}$ act on, is clear from  the context.
Note that $\dd$ and $\dd^*$ both commute with $\Delta$ (that is, $\dd_p\Delta_p=\Delta_{p+1}\dd_{p}$ and $\dd_{p-1}^*\Delta_p=\Delta_{p-1}\dd_{p-1}^*$).
\begin{prop} If $H^p=\{0\}$, then the Laplace operator $\Delta: \form^{p}\to \form^{p}$ is invertible.
\end{prop}
\begin{proof}
The invertibility of the Laplacian is an immediate consequence of the classical Hodge decomposition, which we prove next:
\begin{eqnarray} \label{empint}
\form^p =\nullo\dd_p \oplus \nullo\dd_{p-1}^*.
\end{eqnarray}
In order to prove this decomposition, we first demonstrate that $\nullo\dd_p$ and $\nullo\dd_{p-1}^*$ are orthogonal. Let $u,v\in \form^p$ be such that $\dd u=0$ and $\dd^*v=0$. Since $H^p=\{0\}$, there exists
$w\in\form^{p-1}$ such that $u=\dd w$. Therefore
$$\langle u,v\rangle=\langle \dd w,v\rangle=\langle w,\dd^*v\rangle=0,$$
as desired. Next, we demonstrate that
$$\varphi \in \(\nullo\dd_p \oplus \nullo\dd_{p-1}^*\)^\perp \Rightarrow \varphi=0.$$
Suppose that: (i) $\langle\varphi, u\rangle=0$, $\forall u\in \nullo \dd_p$, that is, for all $p$-forms $u$ such that $u=\dd w$ for some
$w\in\form^{p-1}$; (ii) $\langle\varphi, v\rangle=0$, $\forall v\in \nullo \dd_{p-1}^*$, in particular for all the $p$-forms such that $v=\dd^*z$, for some $z\in\form^{p+1}$.
By using (i), $\langle\varphi, \dd w\rangle=\langle\dd^*\varphi, w\rangle=0$, $\forall w\in\form^{p-1}$, that is, $\dd^*\varphi=0$. Moreover, by using (ii),
$\langle\varphi, \dd^* z\rangle=\langle\dd\varphi, z\rangle=0$, $\forall z\in\form^{p+1}$, that is, $\dd\varphi=0$ $\Rightarrow$ $\varphi=\dd\psi$, for some $\psi\in\form^{p-1}$.
In conclusion, $\langle\varphi,\varphi\rangle=\langle\dd\psi,\varphi\rangle=\langle\psi,\dd^*\varphi\rangle=0$, as desired (in the last step we used that $\dd^*\varphi=0$).
This concludes the proof of \eqref{empint}.

We are now in position of proving the invertibility of $\Delta$. We first prove injectivity: assume that $\Delta u=0$, from which
$$ 0=\langle u,\Delta u\rangle = \langle \dd u, \dd u\rangle + \langle \dd^* u, \dd^* u\rangle,$$
that is $\dd u=0$ and $\dd^* u=0$. In view of \eqref{empint}, this implies that $u=0$ and, therefore, $\Delta$ is injective.
Next, we prove surjectivity:  assume that $u \in (\range \Delta)^\perp$, i.e.
$\langle u,\Delta v\rangle=0$ for all $v$. Then $\langle \Delta u, v\rangle =0$ for all $v$, that is $\Delta u=0$, which implies $u=0$, as we already saw.
\end{proof}

The condition $H^p=\{0\}$ motivates the use of Dirichlet or Neumann boundary conditions, in which cases $H^1=H^2=\{0\}$, so that the Laplacian acting on $1$- and $2$-forms are invertible.

With this notation the Ariza-Ortiz energy can be written as a functional
\begin{eqnarray} \label{HL}
H_\mathrm{AO}:\form^0\times \form^1_\LL\to \R,\; H_\mathrm{AO}(u,\sigma) = \frac{1}{2} \langle \dd u -\sigma, B(\dd u - \sigma)\rangle,
\end{eqnarray}
where $B \in \mathrm{Lin}(\form^1,\form^1)$ is {the linear operator such that, for any $v\in \form^1$, 
$$ (Bv)(e) = (v(e)\cdot \delta e)\,\delta e, $$
and,} given $e=(x,y)$, we denoted $\delta e:=y-x$.
We are now in a position to establish the decomposition of the Ariza-Ortiz energy into elastic and dislocation part.
\begin{thm} \label{energy-decomp}
Let $\Lambda\subset \mathcal L$ be as in \eqref{LN}, and consider the Ariza-Ortiz Hamiltonian \eqref{HL} with Dirichlet or Neumann boundary conditions.
Assume that $q \in \form^2_{\LL}$ satisfies $\dd q =0$. For any $\sigma \in \form^1_\LL$ with the property $\dd \sigma = q$,
the Ariza-Ortiz energy admits the additive decomposition
\begin{equation}
\label{AOdecomp}
H_{\mathrm{AO}}(u,\sigma) = \frac12\langle \dd(u-u_\sigma), B(\dd(u-u_\sigma)\rangle + \frac12\langle \sigma_q, B\sigma_q\rangle,
\end{equation}
where $\sigma_q$ is the minimizer of
$v \mapsto \langle v ,B v\rangle$ on $\form^1$ (rather than on $\form^1_\LL$) subject to the constraint that $\dd v =q$, and $u_\sigma$ is defined as
\begin{equation} \label{uqdef}
u_\sigma = \dd^*\Delta^{-1}(\sigma-\sigma_q).
\end{equation}
If we take Dirichlet boundary conditions, then the minimizer $\sigma_q$ is given by $\sigma_q = Gq$ with
\begin{equation} \label{defG}
G = (1-\dd A^{-1} \dd^* B)\dd ^* \Delta^{-1},
\end{equation}
where $A:{\form^0\to\form^0}$ is the invertible operator $A:=\dd^* B\dd$. \end{thm}

\medskip

For later reference, we note that the adjoint of $G$ can be explicitly written as
\begin{equation} G^*=\Delta^{-1}\dd(1-B\dd A^{-1} \dd^*).\label{Gstar}\end{equation}
As an immediate consequence from~\eqref{AOdecomp} one obtains the equation
\begin{equation} \label{enequiv}
\min_u H_\mathrm{AO}(u,\sigma)= \frac12\langle \sigma_q, B\sigma_q\rangle.
\end{equation}

\begin{proof}
Equation~\eqref{uqdef} implies that $\dd u_\sigma = \dd \dd^*\Delta^{-1}  (\sigma-\sigma_q)$. Moreover, thanks to the commutation relation $\dd \Delta = \Delta \dd$, and
recalling that $\dd \sigma = \dd \sigma_q$, we also find that
$\dd^* \dd\Delta^{-1}  (\sigma-\sigma_q)=\dd^*\Delta^{-1} \dd (\sigma-\sigma_q)=0$. Combining these two identities, we find
\begin{equation} \label{duq}
\dd u_\sigma  =(\dd \dd^* + \dd^* \dd) \Delta^{-1} (\sigma-\sigma_q) = \sigma-\sigma_q.
\end{equation}

Furthermore, as $\sigma_q$ is the solution of a constrained minimization problem there exists a Lagrange parameter $\lambda \in \form^2$ such that $\sigma_q$ satisfies the Euler-Lagrange equation $B\sigma_q = \dd^* \lambda$. Hence, thanks to $\dd^* \dd^*=0$, one obtains
\begin{equation} \label{equil}
\dd^* B \sigma_q  =0.
\end{equation}
Therefore,
\begin{eqnarray*}
H_{{\mathrm{AO}}}(u,\sigma) &=& \frac{1}{2}\langle \dd u - \sigma, B(\dd u-\sigma)\rangle
 = \frac{1}{2}\langle \dd u - \sigma+\sigma_q-\sigma_q, B(\dd u-\sigma+\sigma_q -\sigma_q)\rangle\\
 &\stackrel{\eqref{duq}}{=}& \frac{1}{2}\langle\dd(u-u_\sigma),B\,\dd(u-u_\sigma)\rangle + \frac{1}{2}\langle \sigma_q,B \sigma_q\rangle - \langle u-u_\sigma,\dd^* B\sigma_q\rangle\\
 &\stackrel{\eqref{equil}}{=}& H_{{\mathrm{AO}}}(u-u_\sigma,0) + H_{{\mathrm{AO}}}(0,\sigma_q)
\end{eqnarray*}
which establishes~\eqref{AOdecomp}.

To derive formula~\eqref{defG} we decompose $\sigma_q$ according to the Hodge decomposition \eqref{empint}: $\sigma_q=\varphi+\psi$, with $\dd\varphi=0$ and $\dd^*\psi=0$. We find $\psi=\dd^*\Delta^{-1}q$: in fact, with this position, $\dd\varphi=\dd\sigma_q-\dd\dd^*\Delta^{-1}q=0$, where in the last step we used the fact that $\dd\dd^*\Delta^{-1}q=
(\dd\dd^*+\dd^*\dd)\Delta^{-1}q=q$ (in turn, the fact that $\dd^*\dd\Delta^{-1}q=0$ {follows from the fact that $\dd$ commutes with $\Delta^{-1}$ and $\dd q=0$}). 
In conclusion, $\sigma_q=\varphi+\dd^*\Delta^{-1}q$, with $\dd\varphi=0$. Since $H^2=0$, $\varphi$ is exact and, therefore,
\begin{equation} \label{sigmaqdecomp}
 \sigma_q = \dd u_0+\dd^* \Delta^{-1}q,
\end{equation}
for some $u_0\in\form^1$. Equations~\eqref{sigmaqdecomp} and~\eqref{equil} together imply that
$$ \dd^*B(\dd^* \Delta^{-1} q + \dd u_0)=0.$$
Since $\dd^* B\dd = A$ we obtain that $u_0 = - A^{-1} \dd^*B\dd^* \Delta^{-1}q$ and, therefore,
$\sigma_q = (1-A^{-1}\dd^*B) \dd^* \Delta^{-1} q$, as desired, provided that $A$ is invertible. Finally, the invertibility of $A$ for $\mathcal L$ the FCC lattice
and $\Lambda\subset \mathcal L$ a finite box with Dirichlet boundary conditions is proved in Appendix \ref{app.c0}.
\end{proof}

\section{Proof of Theorem~\ref{lro_simp.thm}} \label{proof_lro.sec}

Our goal is to compute a lower bound on
\begin{equation}\label{uuu}\mathbb{E}_{\beta,\Lambda}({\varphi}) =
\frac{1}{Z_{\beta,\Lambda}}\sum_{\sigma \in \mathcal S}\,\int du\, \e^{-\beta\,(H_\mathrm{AO}(u,\sigma) +W(\dd \sigma ))}\, {\varphi(u)},\end{equation}
{in the two cases that $\varphi(u)=\varphi_{v_0;x}(u)=\cos(u(x)\cdot v_0)$ and $\varphi(u)=\varphi_{v_0;x,y}(u)=\cos((u(x)-u(y))\cdot v_0)$.}
{These functions can be conveniently rewritten in terms of the functions $g,h$ and $\tilde g, \tilde h$}, defined as follows. 
\begin{itemize}
\item
{$g\equiv g_{v_0;x}:= {\bf 1}_{x}v_0$ and $\tilde g\equiv g_{v_0;x,y}:=({\bf 1}_x-{\bf 1}_y)v_0$}, or equivalently 
$${\langle g,u\rangle = u(x)\cdot v_0, \qquad \langle \tilde g,u\rangle=(u(x)-u(y))\cdot v_0.}$$
\item
$\dd^* h =g$ {and $\dd^* \tilde h =\tilde g$}, or equivalently $\langle h,\dd u\rangle = \langle g,u\rangle$ {and 
$\langle \tilde h,\dd u\rangle = \langle \tilde g,u\rangle$.}
\end{itemize}
To show that the equation $\dd^*h = g$ actually admits a solution we can consider a pairwise disjoint collection of {oriented} edges $(e_i)_{i=1\ldots n} \subset E_1$ that form an 
oriented connected path 
$\mathcal P_{{x\to x_{ext}}}$ from $x$ to $x_{ext}$, where $x_{ext}$ is some vertex outside $\Lambda$, which we fix once and for all at a distance $1$ from $\Lambda$.
With such a path we can define
$$h(e)\equiv h_{v_0;x,{x_{ext}}}(e):= \left\{\begin{array}{rl} -v_0 & \text{ if } e\in \mathcal P_{{x\to x_{ext}}},\\
\phantom{-}v_0 & \text{ if } {\bar e}\in\mathcal P_{{x\to x_{ext}}},\\
0 & \text{ else },\end{array}\right.$$
{where in the second line $\bar e$ is the edge with the same vertices as $e$ but opposite orientation. 
Analogously, we let $\tilde h=h_{v_{0};x,y}$ and note that, with this choice, $\dd^*\tilde h=\tilde g$.} 
In terms of these definitions, for any $\sigma\in\form^1_\LL$,
\begin{equation} \label{obsrep.eq}\begin{split} 
&{\varphi_{v_0;x}(u)=\cos\langle g,u\rangle=\cos\langle h, \dd u\rangle=\cos\langle h, \dd u-\sigma\rangle,}\\
&{\varphi_{v_0;x,y}(u)=\cos\langle \tilde g,u\rangle=\cos\langle \tilde h, \dd u\rangle=\cos\langle \tilde h, \dd u-\sigma\rangle,}\end{split}
\end{equation}
{because both $\langle h,\sigma\rangle$ and $\langle \tilde h,\sigma\rangle$ are equal to $0$ mod $2\pi$.}

Thanks to the decomposition \eqref{AOdecomp}, {in the cases of interest} \eqref{uuu} can be rewritten as
\begin{equation}\label{uuuu}\mathbb{E}_{\beta,\Lambda}({\varphi_{v_0;x}}) =
\frac{1}{Z_{\beta,\Lambda}}\sum_{\sigma \in \mathcal S}\,\int du\, \e^{-\frac{\beta}2 \langle \dd(u-u_\sigma), B\dd(u-u_\sigma)\rangle} e^{-\beta W(q)-\frac{\beta}2\langle \sigma_q, B\sigma_q\rangle}
\, {\cos\langle h,\dd u-\sigma\rangle},\end{equation}
where $q=\dd\sigma$, {and similarly for $\mathbb{E}_{\beta,\Lambda}(\varphi_{v_0;x,y})$, with $h$ replaced by $\tilde h$.} 
Recalling \eqref{duq}, we can rewrite $\dd u-\sigma=\dd (u-u_\sigma) -\sigma_q$, with $\sigma_q=Gq$, so that, renaming $u-u_{\sigma}\equiv u'$,
\begin{equation}\label{uuuuh}\mathbb{E}_{\beta,\Lambda}(\varphi_{v_0;x}) =
\frac{1}{Z_{\beta,\Lambda}}\sum_{q\in \form^{2}_*}\,\int du'\, \e^{-\frac{\beta}2 \langle \dd u', B \dd u'\rangle} e^{-\beta W(q)-\frac\beta2\langle q, G^*B Gq\rangle }
\, {\cos\langle h,\dd u'-Gq\rangle},\end{equation}
where $\form^{2}_*=\{q\in \form^2_\mathcal L: \, \dd q=0\}$ is the set of closed, lattice-valued, $2$-forms {satisfying Dirichlet boundary conditions on $\Lambda$};
{again, $\mathbb{E}_{\beta,\Lambda}(\varphi_{v_0;x,y})$ admits an analogous representation, with $h$ replaced by $\tilde h$.}
Note that the probability measure in the right side of \eqref{uuuuh}
is factorized: it is the product of
a Gaussian measure $\mathbb P_{\beta,\Lambda}^{sw}$ on $u'$ (the {\it spin wave} part of the measure) times a discrete measure $\mathbb P_{\beta,\Lambda}^{dis}$ on the dislocation cores $q$.
This factorization property is due to the quadratic nature of the Ariza-Ortiz model, and makes our statistical mechanics version of the Ariza-Ortiz model reminiscent of the
Villain model for classical rotators. Of course, the partition function inherits the same factorization property:
$Z_{\beta,\Lambda} = Z^{sw}_{\beta,\Lambda}\,Z^{dis}_{\beta,\Lambda}$, with
\begin{equation}
Z^{sw}_{\beta,\Lambda} = \int du'\, \e^{-\frac{\beta}2\langle \dd u',B\dd u'\rangle}, \qquad
Z^{dis}_{\beta,\Lambda}= \sum_{q\in\form^2_*}\e^{-\beta W(q)-\frac\beta2\langle q, G^*B Gq\rangle}.
\end{equation}
Plugging this representation in \eqref{uuuuh} {and in its analogue for $\mathbb E_{\beta,\Lambda}(\varphi_{v_0;x,y})$}, 
and noting that $\mathbb P_{\beta,\Lambda}^{sw}$ and $\mathbb P_{\beta,\Lambda}^{dis}$ are even,
we find
\begin{equation}\begin{split} & {c_{\beta,\Lambda}(v_0;x)=}\mathbb{E}_{\beta,\Lambda}({\varphi_{v_0;x}}) =\mathbb E_{\beta,\Lambda}^{sw} (\cos\langle h,\dd u'\rangle)\,
\mathbb E_{\beta,\Lambda}^{dis} (\cos\langle h, Gq\rangle),\\
& {c_{\beta,\Lambda}(v_0;x,y)=}\mathbb{E}_{\beta,\Lambda}({\varphi_{v_0;x,y}}) =\mathbb E_{\beta,\Lambda}^{sw} (\cos\langle \tilde h,\dd u'\rangle)\,
\mathbb E_{\beta,\Lambda}^{dis} (\cos\langle \tilde h, Gq\rangle).\end{split}\end{equation}
Theorem \ref{lro_simp.thm} is a consequence of the following more refined version thereof.

\begin{thm} \label{lro.thm}
Given $v_0\in\mathcal L^*$, there exist positive constants $C_0,C,c,\beta_0, {r_0}$ such that, if $\beta\ge \beta_0$ {and $|x-y|\ge r_0$,}
\begin{equation}\label{sw.me}\begin{split}
& \lim_{\Lambda\to\mathcal L}\mathbb E_{\beta,\Lambda}^{sw} (\cos\langle h,\dd u'\rangle)=e^{-C_0/(2\beta)},\\
& \lim_{\Lambda\to\mathcal L}\mathbb E_{\beta,\Lambda}^{sw} (\cos\langle \tilde h,\dd u'\rangle)=e^{-C_0/\beta}\Big(1+O\big(\scalebox{1.2}{$\frac{\log|x-y|}{|x-y|}$}\big)\Big)
,\end{split}\end{equation}
where {$h=h_{v_0;x,x_{ext}}$ and $\tilde h=h_{v_0;x,y}$}; moreover,
\begin{equation}\label{dis.me} \begin{split}
& \liminf_{\Lambda\to\mathcal L}\mathbb E_{\beta}^{dis} (\cos\langle h, Gq\rangle)\ge {e^{-Ce^{-c \beta}}},\\
& \liminf_{\Lambda\to\mathcal L}\mathbb E_{\beta}^{dis} (\cos\langle \tilde h, Gq\rangle)\ge {e^{-Ce^{-c \beta}}},\end{split}\end{equation}
so that ${\liminf}_{{\Lambda\to\mathcal L}}c_{\beta,\Lambda}(v_0;x)\ge \exp\{- C_0/(2\beta)+O(e^{-c\beta})\}$
and $\liminf_{|x-y|\to\infty}{\liminf}_{{\Lambda\to\mathcal L}}$ $c_{\beta,\Lambda}(v_0;x,y)\ge \exp\{- C_0/\beta+O(e^{-c\beta})\}$.
\end{thm}
The rest of the section is devoted to the proofs of \eqref{sw.me} and \eqref{dis.me}.

\subsection{The spin wave contribution to the two-point function (proof of \eqref{sw.me})}\label{sec.swme}
Recalling the definitions $A=\dd^* B\dd$ and $g=\dd^*h{={\bf 1}_{x}\cdot v_0}$, we find that the spin wave contribution to the {one}-point function is
\begin{eqnarray}
&& \mathbb E_{\beta,\Lambda}^{sw} (\cos\langle h,\dd u'\rangle)= \frac1{Z^{sw}_{\beta,\Lambda}} \int\, du \exp\Big\{\scalebox{1.2}{$-\frac\beta2$}\langle\dd u',B\,\dd u'\rangle+ i \langle h,\dd u'\rangle\Big\}\nonumber \\
 &&=  \frac1{Z^{sw}_{\beta,\Lambda}} \int du'\exp\Big\{\scalebox{1.2}{$-\frac\beta2$}\langle u',A\,u'\rangle+ i \langle g,u'\rangle\Big\}=
 \exp\Big\{\scalebox{1.2}{$-\frac1{2\beta}$}\langle g, A^{-1}g\rangle\Big\}.\label{eq:27}
\end{eqnarray}
As proved in Appendix \ref{app.c0}, the thermodynamic limit of the right side can be explicitly written in Fourier space:
\begin{equation}\label{upb}\lim_{\Lambda \to \mathcal L}\langle g, A^{-1}g\rangle=\int_{\mathcal B}\frac{dk}{|\mathcal B|}\, \hat g(-k)\cdot \hat A^{-1}(k) \hat g(k),\end{equation}
where $\mathcal B=\{\xi_1m_1+\xi_2 m_2+\xi_3 m_3: \ \xi_i\in[0,1)\}$ is the Brillouin zone (recall that $m_1,m_2,m_3$ are the basis vectors of $\mathcal L^*$, see Appendix \ref{app.FCC}), $|\mathcal B|$ is its volume,
\begin{equation}\label{hatgk}\hat g(k)=\sum_{z\in\mathcal L} g(z) e^{ik\cdot z}=v_0 {e^{ik\cdot x}},\end{equation}
and, if $\Pi_l=b_l\otimes b_l$,
\begin{equation} \label{eqAk}\hat A(k)=2\sum_{l=1}^6 \Pi_{l}\,(1-\cos(k\cdot b_l)).\end{equation}
In Appendix \ref{app.c0} we prove that $\hat A(k)$ is singular iff $k\in \mathcal L^*$ and that, if $k$ is close to $0$, $\hat A(k)=\hat A_0(k)(1+O(k))$, with
\begin{equation} \label{c0}c_0k^2\mathds 1\le \hat A_0(k)\le \frac32k^2\mathds 1,\end{equation}
where the positive constant $c_0$ can be chosen, e.g., to $c_0=(3-\sqrt5)/4$.
We remark that this bound depends critically on the structure of the underlying lattice: changing FCC into cubic does not preserve the property that
$\hat A(k)$ behaves qualitatively like the Laplacian $k^2$ at low momenta.
The inverse operator reads:
\begin{equation}\hat A(k)^{-1}=\hat A_0(k)^{-1}(1+O(k)),\label{inv.}\end{equation}
so that the {integral in the right side of \eqref{upb} is convergent and we can rewrite}
$${\lim_{\Lambda\to\mathcal L}\langle g,A^{-1}g\rangle=C_0:=\int_{\mathcal B}\frac{dk}{|\mathcal B|}\,v_0\cdot \hat A^{-1}(k)v_0.}$$
Note that the fact that $C_0$ is finite crucially depends on the fact that we are in three (or, better, in more than two) dimensions: in two dimensions its analogue would be logarithmically divergent.
{This concludes the proof of the first of \eqref{sw.me}. The proof of the second of \eqref{sw.me} is analogous: a repetition of the previous computation implies that 
$\mathbb E_{\beta,\Lambda}^{sw} (\cos\langle \tilde h,\dd u'\rangle)=\exp\{-\frac1{2\beta}(\tilde g,A^{-1}\tilde g)\}$, with}
\begin{equation}\label{upbbis}\begin{split}{\lim_{\Lambda \to \mathcal L}\langle \tilde g, A^{-1}\tilde g\rangle}&{=\int_{\mathcal B}\frac{dk}{|\mathcal B|}\, v_0\cdot \hat A^{-1}(k)v_0|e^{ikx}-e^{iky}|^2}\\
&=2C_0
-2\int_{\mathcal B}\frac{dk}{|\mathcal B|}\, v_0\cdot \hat A^{-1}(k)v_0 e^{-ik(x-y)},\end{split} \end{equation}
{which leads to the second of \eqref{sw.me}, thanks to the fact that}
\begin{equation}\label{logbbis}\Big| \int_{\mathcal B}\frac{dk}{|\mathcal B|}\,v_0\cdot \hat A^{-1}(k)v_0\, e^{-ik\cdot (x-y)}\Big|\le c_1
\frac{{1+}\log|x-y|}{|x-y|},\end{equation}
for a suitable constant $c_1>0$, see Appendix \ref{app.c0} for details.

\subsection{The dislocation contribution to the two-point function (proof of \eqref{dis.me})}\label{sec.disme}

We now want to bound from below the dislocation contribution to the {one}-point function, namely
$$ \mathbb E_{\beta}^{dis} (\cos\langle h, Gq\rangle)=\frac1{Z^{dis}_{\beta,\Lambda}}\sum_{q\in\form^2_*} e^{-\beta W(q)-\frac\beta2\langle q,G^* BGq\rangle}\cos\langle h, Gq\rangle,$$
{and similarly for the two-point function.}
Recall that the probability weight $e^{-\beta W(q)}$ is of factorized form, $e^{-\beta W(q)}=\prod_{f}e^{-\beta w(q(f))}$, where the product runs over the faces that have non zero intersection with $\Lambda$ and
$w(q(f))$ $\ge w_0 |q(f)|^2$ for some positive $w_0$.
Note that this weight can be equivalently rewritten as
\begin{equation}\label{qf}e^{-\beta W(q)}=\Big[\prod_{f}\lambda(q(f))\Big]e^{-\beta \frac{w_0}{2}\langle q,q\rangle},\end{equation}
where $\lambda(x)=\exp\{-\beta[w(x)-{w_0}|x|^2/2]\}\le \exp\{-\beta w_0|x|^2/2\}$.
The observable of interest can be rewritten as $Z^{dis}_{\beta,\Lambda}(h)/Z^{dis}_{\beta,\Lambda}(0)$, with
$$Z^{dis}_{\beta,\Lambda}(h)=\sum_{q\in\form^2_*}e^{-\beta W(q)-\frac{\beta}{2} \langle q,\, G^* B G q\rangle+ i \langle h, Gq\rangle}.$$
The goal is to find a lower bound on $Z^{dis}_{\beta,\Lambda}(h)/Z^{dis}_{\beta,\Lambda}(0)$, of lower order than the spin wave contribution.
We now perform a sine-Gordon transformation: we introduce the Gaussian measure $\mu_\beta(d\phi)$ with covariance $\beta\, (G^*BG+w_0\mathds 1) {>0}$, so that
$$e^{-\frac{\beta}2 \langle q, (G^* B G+w_0\mathds 1) q\rangle}=\int \mu_\beta(d\phi) e^{i\langle q,\phi\rangle},$$
and rewrite
$$Z_{\beta,\Lambda}^{dis}(h)=\sum_{q\in\form^2_*} \Big[\prod_{f} \lambda(q(f))\Big] \int \mu_\beta(d\phi)e^{i\langle q,\phi+ G^* h\rangle}.$$

{\begin{rem}
The presence of the $w_0\mathds 1$ term in the covariance $\beta\, (G^*BG+w_0\mathds 1)$ is crucial for making it positive definite, rather than just 
non-negative. In fact, $G^*BG$ has 
null directions. To see this, take $\sigma\in \mathcal C^1_{\mathcal L}$ to be $+b_4$ resp. $-b_4$ on the edge $(0,b_1)$ resp. $(b_1,0)$, and zero otherwise,
so that $B\sigma=0$ but $q:=\dd\sigma\neq 0$. With these definitions, it is easy to check that 
$(q,G^*BGq)=(\sigma_q,B\sigma_q)=0$: in fact, $\sigma_q$ is the minimizer of $(v,Bv)=\|Bv\|^2$
among the $v\in \mathcal C^1$ such that $\dd v=q$; taking $v=\sigma$, we have
$Bv=0$, so $B\sigma_q=0$.\end{rem}}

We now perform the sum over $q$ in two steps: we first fix the support of $q$ and then sum over the charge configurations compatible with that support:
\begin{equation}\begin{split}Z_{\beta,\Lambda}^{dis}(h)&=\sum_{X\subseteq \Lambda_2}\sum_{\substack{q\in\form^2_*:\\ \supp(q)=X}} \Big[\prod_{f\in\Lambda_2} \lambda(q(f))\Big] \int \mu_\beta(d\phi)e^{i\langle q,\phi+G^* h\rangle}\\ &\equiv \int \mu_\beta(d\phi)\Big[
\sum_{X\subseteq \Lambda_2} K(X,\phi+G^* h)\Big].\end{split}\nonumber
\end{equation}
For short, we shall write $K(X,\phi+ G^* h)=K(X)$.
Note that $K(X_1 \cup X_2) = K(X_1)K(X_2)$ if $X_1, X_2 \subset \Lambda_2$ are {\it disconnected}, i.e., if there is no $3$-cell whose boundary contains both an element of $X_1$ and an element
of $X_2$;
the key thing to observe is that
in such a situation, if we let $q=q_1+q_2$ with $\supp(q_1)=X_1$ and $\supp(q_2)=X_2$, the constraint $d(q_1+q_2)=0$ `factorizes' in $dq_1=dq_2=0$. Note that 
this factorization into {\em locally neutral} contributions is another point where the condition that we are in more than two dimensions enters crucially.
Given $X$, we let $X_1,\ldots, X_n$ be its maximally connected components
and note that
$$K(X)=\kappa(X_1)\cdots \kappa(X_n),$$
where
$$\kappa(X)=\mathds 1(X\,{\text{is}}\,{\text{connected}})\sum_{\substack{q\in\form^2_*:\\ \supp(q)=X}} \Big[\prod_{f\in X} \lambda(q(f))\Big] e^{i\langle q,\phi+G^* h\rangle}$$
The upper bound on $\lambda$ implies that
\begin{equation}\label{upkappa}|\kappa(X)| \leq  \Big(\sum_{\substack{b\in\mathcal L:\\ b\neq 0}}e^{-\frac{\beta}2w_0|b|^2}\Big)^{|X|}.\end{equation}
We use the factorization property of $K$ to rewrite $Z^{dis}_{\beta,\Lambda}(h)$ as
$$Z^{dis}_{\beta,\Lambda}(h)= \int \mu_\beta(d\phi)\Big[\sum_{n\ge 0} \frac1{n!}
\sum_{X_1,\ldots, X_n \subseteq \Lambda_2} \kappa(X_1)\cdots \kappa(X_n) \delta(X_1,\ldots,X_n)\Big],$$
where $\delta(X_1,\ldots,X_n)=\prod_{1\le i<j\le n}\delta(X_i,X_j)$, and $\delta(X,Y)=1$ if $X$ and $Y$ are disconnected, and $=0$ otherwise. As well known, see, e.g., \cite[Proposition 5.3]{FV},
\begin{eqnarray} && \sum_{n\ge 0} \frac1{n!}
\sum_{X_1,\ldots, X_n \subseteq \Lambda_2} \kappa(X_1)\cdots \kappa(X_n) \delta(X_1,\ldots,X_n)=\nonumber\\
&& \label{eq:34}=\exp\Big\{
\sum_{n\ge 1}
\sum_{X_1,\ldots, X_n \subseteq \Lambda_2} \kappa(X_1)\cdots \kappa(X_n) \varphi(X_1,\ldots,X_n)\Big\},\end{eqnarray}
where $\varphi(X_1,\ldots,X_n)$ is the Ursell function: if $G_n$ is the complete graph on the vertex set $\{1,\ldots,n\}$,
$$ \varphi(X_1,\ldots,X_n)= \frac1{n!}\sum_{\substack{G\subseteq G_n\\
{\rm connected}}}\prod_{\{i,j\}\in G}(\delta(X_i,X_j)-1),\qquad {\rm for}\quad n>1,$$
while $\varphi(X_1)=1$, for $n=1$. Note that $\kappa(X_1)\cdots \kappa(X_n)\varphi(X_1,\ldots,X_n)$ is non-zero only if $Y=X_1\cup\cdots \cup X_n$ is connected.
The sums in \eqref{eq:34} are absolutely convergent, uniformly in $\Lambda$, provided that there exists a positive function $a(X)$, independent of $\Lambda$, such that, for any fixed,
connected, non-empty $X_*\subseteq\Lambda_2$,
$$\sum_{X\subseteq \Lambda_2}|\kappa(X)|e^{a(X)}(1-\delta(X,X_*))\le a(X_*),$$
see \cite[Theorem 5.4]{FV}. In our case, if $\beta$ is sufficiently large, thanks to the upper bound on $\kappa(X)$, eq.\eqref{upkappa}, we can choose $a(X)=e^{-\beta w_0/4} |X|$.
We now insert the definition of $\kappa$ in \eqref{eq:34} and rewrite it as
$$\eqref{eq:34}=\exp\Big\{\sum_{q\in\form^2_*}z(\beta,q)e^{i\langle q,\phi+G^* h\rangle}\Big\}$$
with
\begin{equation} z(\beta,q)=\sum_{n\ge 1} \sum_{\substack{q_1,\ldots, q_n\in\form^2_*:\\ q_1+\cdots+q_n=q}}  \Big[\prod_{i=1}^n\mathds 1({X_i\,{\text{is}\, {\text{connected}}}})\Big]
\Big[\prod_{i=1}^n\Big(\prod_{f\in X_i} \lambda(q_i(f))\Big)\Big] \varphi(X_1,\ldots,X_n),\label{zbetaq}\end{equation}
where in the right side $X_i:=\supp(q_i)$. Using the fact that $\lambda(x)$ is exponentially small, as well as the fact that the Ursell function decays exponentially to zero at large distances, we get
that, for $\beta$ large enough,
\begin{equation}\label{CE.1}|z(\beta,q)|\le e^{-\frac\beta{4} w_0 \|q\|_1}e^{-\frac{\beta}{8}w_0}|\supp(q)|,\end{equation}
see Appendix \ref{app.CE} for a proof.
Note also that $z(\beta, q)$ is zero unless $q$ has connected support.
Putting things together we find:
\begin{equation}\nonumber \begin{split}
\frac{Z_{\beta,\Lambda}^{dis}(h)}{Z_{\beta,\Lambda}^{dis}(0)}=\int m_\beta(d\phi) \exp\Big\{-\sum_{q\in\form^2_*}z(\beta,q)\big[&\cos\langle q,\phi\rangle\big(1-\cos\langle q, G^* h\rangle\big)\\
+&\sin\langle q,\phi\rangle\sin\langle q, G^* h\rangle\big]\Big\},\end{split}\end{equation}
where
$$m_\beta (d\phi)=\frac{\mu_\beta(d\phi) e^{\sum_{q	\in\form^2_*}z(\beta,q)\cos\langle q,\phi\rangle }}{\int \mu_\beta(d\phi) e^{\sum_{q\in\form^2_*}z(\beta,q)\cos\langle q,\phi\rangle }}.$$
We now apply Jensen's inequality, i.e., $\int m_\beta(d\phi) \exp\big\{(\cdot)\big\}\ge \exp\big\{\int m_\beta(d\phi)(\cdot)\big\}$, and find (noting that $\int m_\beta(d\phi)\sin\langle q,\phi\rangle=0$ and $|\int m_\beta(d\phi)\cos\langle q,\phi\rangle|\le 1$),
\begin{equation}\label{fracZhZ0}
\frac{Z_{\beta,\Lambda}^{dis}(h)}{Z_{\beta,\Lambda}^{dis}(0)}\ge \exp\Big\{-\sum_{q\in\form^2_*}|z(\beta,q)|\big(1-\cos\langle q, G^* h\rangle\big)\Big\}.\end{equation}
We now need to manipulate $\langle q, G^* h\rangle$. Proposition \ref{lat-Poincare} implies that there exists an $\mathcal L$-valued 1-form $n_q$ such that $\dd n_q=q$.
Recall that
\begin{enumerate}
\item The support of $n_q$ is contained in $B(q)$, the smallest parallelepiped
containing the support of $q$,
\item The maximum of $|n_q|$ is bounded in terms of the $2$-norm of $q$, as follows: $\|n_q\|_\infty\le c\|q\|_2^4$ for some positive $c$.
\end{enumerate}
Therefore, recalling the definition of $G^*$, see \eqref{Gstar}, we get
\begin{equation}\langle q,G^*h\rangle=\langle\dd n_q,\dd \Delta^{-1}(1-B\dd A^{-1} \dd^*)h\rangle=\langle n_q,\dd^*\dd\Delta^{-1}(1-B\dd A^{-1} \dd^*)h\rangle.\end{equation}
Now recall that $\dd^*\dd=\Delta-\dd\dd^*$, so that
\begin{equation}\langle q,G^*h\rangle=\langle n_q,h\rangle-\langle n_q,B\dd A^{-1} \dd^*h\rangle-\langle n_q,\dd\dd^*\Delta^{-1}(1-B\dd A^{-1} \dd^*)h\rangle.\end{equation}
Now the first term in the right side is an integer multiple of $2\pi$, and can be dropped for the purpose of computing the cosine. The last term, by using that $\dd^*$ commutes with $\Delta$, equals
$$-\langle n_q,\dd \Delta^{-1}(\dd^*-\dd^*B\dd A^{-1}\dd^*)h\rangle,$$
which is zero, simply because $\dd^*B \dd A^{-1}=1$. We are left with the second term, which can be rewritten in terms of $g=\dd^*h$; in conclusion:
\begin{equation}\langle q,G^*h\rangle=-\langle n_q,B\dd A^{-1} g\rangle\quad {\rm mod}\,2\pi.\end{equation}
If we now plug this into \eqref{fracZhZ0} and bound $1-\cos x\le x^2/2$, we find
$$\frac{Z_{\beta,\Lambda}^{dis}(h)}{Z_{\beta,\Lambda}^{dis}(0)}\ge \exp\{-\frac12\sum_{q\in\form^2_*}|z(\beta,q)| \langle\epsilon, n_q\rangle^2\}, $$
where $\epsilon=B\dd A^{-1}g$. Recalling that $\|n_q\|_\infty\le c\|q\|_2^4$ and supp $n_q\subseteq B(q)$, we find
$$\langle\epsilon, n_q\rangle^2\le  c^2\|q\|_2^8\sum_{e,e'\in B(q)} \epsilon(e) \epsilon(e') .$$
Using this bound, we get:
$$\frac{Z_{\beta,\Lambda}^{dis}(h)}{Z_{\beta,\Lambda}^{dis}(0)}\ge \exp\{-\frac{c^2}2\sum_{e,e'\in\Lambda_1}\epsilon(e)\epsilon(e') \sum_{q\in\form^2_*}|z(\beta,q)| \|q\|_2^8 \mathds 1(B(q)\ni e,e')\}. $$
Now, using the bound \eqref{CE.1} on $z(\beta,q)$, we get (see Appendix \ref{app.CE} for details)
\begin{equation}\label{CE.2}\sum_{q\in\form^2_*}|z(\beta,q)|\, \|q\|_2^8\ \mathds 1(B(q)\ni e,e')\le e^{-\frac{\beta}{8}w_0[1+ {\rm dist}(e,e')]},\end{equation}
provided that $\beta$ is large enough. Therefore,
$$\frac{Z_{\beta,\Lambda}^{dis}(h)}{Z_{\beta,\Lambda}^{dis}(0)}\ge \exp\{-\frac{c^2}2e^{-\frac\beta{8} w_0}\sum_{e,e'\in\Lambda_1}\epsilon(e)\epsilon(e') e^{-\frac{\beta}{8}w_0 {\rm dist}(e,e')}\}\ge e^{-c(\beta)
\langle\epsilon,\epsilon\rangle}, $$
where in the last step we used Young's inequality and
$$c(\beta)=\frac{c^2}2e^{-\frac\beta{8} w_0} \sum_{e'\in E_1}e^{-\frac{\beta}{8}w_0 {\rm dist}(e,e')}.$$
Finally, we note that
$$\langle\epsilon,\epsilon\rangle=\langle g,A^{-1}g\rangle$$
that, combined with \eqref{eq:27}, implies the desired estimate, {the first of} eq.\eqref{dis.me}. {A step by step repetition of the argument with $h$ and $g$ replaced by $\tilde h$ and $\tilde g$, respectively, 
leads to the second of \eqref{dis.me}.}

\section{Proof of Theorem~\ref{thm.RS}}\label{sec.RS}

In this section we compute the energy of a dislocation dipole $q^n_{\mathrm{dip}}$, see \eqref{disl_dist}, asymptotically as $n\to\infty$, and
the energy of two parallel arrays of dipoles $q^{M,n,m}_{\mathrm{grain}}$, see \eqref{wall_dist}, asymptotically as $M\to\infty$ first, then $n\to\infty$, then $m\to\infty$.

Note that, also in two dimensions, the Ariza-Ortiz energy satisfies  the additive decomposition property \eqref{AOdecomp}, from which we get
\begin{equation}\label{beh}\min \{H_{\mathrm{AO}}(u,\sigma):\ \dd\sigma=q\}=\frac12\langle q,G^* B Gq\rangle,\end{equation}
see \eqref{enequiv} (with some abuse of notation, in this section we denote by $A,B,G$ the analogues for the triangular lattice of the operators $A,B,G$ introduced in Sect.\ref{HD} for the case of the FCC lattice).

\subsection{The energy of a dipole (proof of \eqref{dip.0}).}
Let $\sigma^n_{\mathrm{dip}}$ be such that $\dd\sigma^n_{\mathrm{dip}}=q^n_{\mathrm{dip}}$. Then
$$E_\mathrm{dip}(n) = \min\{ H(u,\sigma^n_\mathrm{dip}) \; : \; u \in \form^0\}.$$
The optimal 0-form $u$ satisfies $A u = \dd^* B \sigma_\mathrm{dip}^n$ and consequentially
\begin{eqnarray} \nonumber
E_\mathrm{dip}(n) &=& \frac{1}{2}\left \langle \sigma_\mathrm{dip}^n, B \sigma_\mathrm{dip}^n \right \rangle - \left \langle \sigma_\mathrm{dip}^n, B\dd A^{-1} \dd^*B \sigma_\mathrm{dip}^n\right \rangle\\ &+& \frac{1}{2} \left \langle \dd  A^{-1} \dd^*B \sigma_\mathrm{dip}^n, B \dd  A^{-1} \dd^*B \sigma_\mathrm{dip}^n \right\rangle
\label{energy} \\
&= & \frac{1}{2} \left \langle \sigma_\mathrm{dip}^n, B \sigma_\mathrm{dip}^n\right\rangle-\frac{1}{2} \left\langle \dd^* B \sigma_\mathrm{dip}^n, A^{-1} \dd^*B \sigma^n_\mathrm{dip}\right\rangle.\nonumber
\end{eqnarray}

We choose
\begin{equation} \label{defsigdip.eq}
 \sigma_\mathrm{dip}^n(x,x+b_l) = \begin{cases} b_1 & \text{ if $l =2$ and } x = j\,b_1 \text{ for some } j \in \{1,\ldots, n\},\\
-b_1 & \text{ if $l=3$ and } x= j\,b_1 -b_3\text{ for some } j \in \{1,\ldots, n\},\\
0 & \text{ else.}
\end{cases}
\end{equation}
It is a simple exercise to check that $\dd\sigma^n_\mathrm{dip}=q^n_\mathrm{dip}$ is satisfied.
For a visualization of the support of the slip-field $\sigma_\mathrm{dip}^n$ see Fig.~\ref{fig.app}.

We compute the energy in \eqref{energy} by using this $\sigma^n_\mathrm{dip}$. We start by computing $B\sigma^n_\mathrm{dip}$:
$$B\sigma^n_\mathrm{dip}(x,x+b_l)=\begin{cases}
-\frac12 b_2 & \text{ if $l=2$ and } x=jb_1, \text{ for some } j\in\{1,\ldots,n\},\\
+\frac12 b_3 & \text{ if $l=3$ and } x=jb_1-b_3, \text{ for some } j\in\{1,\ldots,n\},\\
0 & \text{ else} \end{cases}$$
which implies that the first term in the right side of \eqref{energy} satisfies (recall that the number of red bonds equals $2n$)
$$ \left\langle \sigma^n_\mathrm{dip}, B\sigma^n_\mathrm{dip}\right\rangle = \frac{n}{2}.$$
In order to explicitly compute the second term in the right side of \eqref{energy} in the thermodynamic limit, it is convenient to fix a convention for the Fourier transform: given functions $u,\sigma,q$ on vertices, edges, faces of the
infinite triangular lattice, respectively, we let
\begin{eqnarray*}
& u(x)=\int_{\mathcal B} \frac{dk}{|\mathcal B|}\hat u(k) e^{-ik\cdot x}, & x\in\mathcal T,\\
& \sigma(x,x+b_l)=\int_{\mathcal B} \frac{dk}{|\mathcal B|}\hat \sigma(k,l) e^{-ik\cdot x}, & x\in\mathcal T,\quad{\rm with}\quad l=1,2,3,
\end{eqnarray*}
where
$$\mathcal B=\left\{k \in \R^2 \; : \;k\cdot b_l \in [0,2\pi),\  l =1,2\right\}$$
is the (first) Brillouin zone, and $|\mathcal B|=8\pi^2/\sqrt3$ its area. With this notation, the thermodynamic limit of the second term in the right side of \eqref{energy} can be written as
\begin{equation}\label{d*Bs}\lim_{\Lambda\to\mathcal T}\langle \dd^*B\sigma^n_{\mathrm{dip}}, A^{-1}\dd^*B\sigma^n_{\mathrm{dip}}\rangle=
\int_{\mathcal B}\frac{dk}{|\mathcal B|} \widehat{\dd^* B \sigma_\mathrm{dip}^n}(-k) \hat A^{-1}(k) \widehat{\dd^* B \sigma_\mathrm{dip}^n}(k), \end{equation}
where $\hat A(k)$ is the Fourier symbol of $A$. Note that, for any 0-form $u$,
$$Au(x)=\sum_{l=1}^3\big(2u(x)-u(x+b_l)-u(x-b_l)\big) \Pi_l, $$
where $\Pi_l=b_l\otimes b_l$ is the projector in direction $b_l$. Therefore, passing to Fourier space, we get
\begin{equation}\begin{split}\hat A(k)&=2\sum_{l=1}^3\Pi_l (1-\cos(k\cdot b_l))\\
&=\begin{pmatrix}3-2\cos k_1-\cos\frac{k_1}{2}\cos\frac{\sqrt3 k_2}{2} & \sqrt3 \sin\frac{k_1}{2}\sin\frac{\sqrt3 k_2}{2}\\
 \sqrt3 \sin\frac{k_1}{2}\sin\frac{\sqrt3 k_2}{2} & 3-3\cos\frac{k_1}{2}\cos\frac{\sqrt3 k_2}{2}\end{pmatrix},\end{split}\label{Ak}\end{equation}
which implies
$$\det \hat A(k)=3\Big(\cos\frac{k_1}{2}- \cos\frac{\sqrt3 k_2}{2}\Big)^2+6\Big(1-\cos\frac{k_1}{2}\cos\frac{\sqrt3 k_2}{2}\Big)\big(1-\cos k_1\big)\ge 0,$$
and $\det \hat A(k)=0$ $\Leftrightarrow$ $k=0$ mod $\mathcal L^*$.

In order to compute $\widehat{\dd^* B \sigma_\mathrm{dip}^n}(k)$ in \eqref{d*Bs}, we first note that
$$ \widehat{B\sigma^n_\mathrm{dip}}(k,l) = \begin{cases} -\frac{e^{ik_1}}2 \frac{e^{ink_1}-1}{e^{ik_1}-1}\,b_2 & \text{ if } l=2,\\[0.5em]
+\frac{e^{i{k_1}}}2 \frac{e^{ink_1}-1}{e^{ik_1}-1} e^{-ik\cdot b_3}\,b_3 & \text{ if } l=3,\\[0.5em]
0 & \text{ else.} \end{cases}$$
Moreover, for any $1$-form $f$,
\begin{eqnarray*}
\dd^* f(x) = \sum_{l=1}^3 (-f(x,x+b_l)+f(x-b_l,x)) = \frac{1}{|\mathcal B|} \int_{k \in \mathcal B}dk\, e^{-ik\cdot x} \sum_{l=1}^3(e^{ik\cdot b_l}-1) \hat f (k,l),
\end{eqnarray*}
which implies
\begin{eqnarray*}
\widehat{\dd^* B \sigma_\mathrm{dip}^n}(k) = -\frac{e^{ik_1}}{2} \frac{e^{ink_1}-1}{e^{ik_1}-1} \left[(e^{ik\cdot b_2}-1)b_2- (1-e^{-ik\cdot b_3})b_3\right].
\end{eqnarray*}
In conclusion,
\begin{eqnarray}\nonumber && \lim_{\Lambda\to\mathcal T}\langle  \dd^*B\sigma^n_\mathrm{dip},\, A^{-1} \dd^* B\sigma^n_\mathrm{dip}\rangle
= \frac14\int_{\mathcal B} \frac{dk}{|\mathcal B|}\, \left| \frac{e^{ink_1}-1}{e^{ik_1}-1}\right|^2\cdot\\
&&\qquad \cdot  \left[(e^{-ik\cdot b_2}-1)b_2- (1-e^{ik\cdot b_3})b_3\right]
 \cdot \hat A(k)^{-1}  \left[(e^{ik\cdot b_2}-1)b_2- (1-e^{-ik\cdot b_3})b_3\right] \nonumber \\
&&\hskip4.93truecm=\frac14\int_{\mathcal B} \frac{dk}{|\mathcal B|}\, \frac{1-\cos(k_1n)}{1-\cos k_1}\,F(k),\label{FF}
\end{eqnarray}
with
\begin{equation}F(k):= \left[(e^{-ik\cdot b_2}-1)b_2- (1-e^{ik\cdot b_3})b_3\right] \hat A^{-1}(k)
 \left[(e^{ik\cdot b_2}-1)b_2- (1-e^{-ik\cdot b_3})b_3\right].\label{eF}\end{equation}
In the vicinity of the singularity, letting $\Pi_k=k\otimes k /k^2$ be the projector in direction $k$,
$$\hat A(k)=\frac38 k^2(\mathds 1+2\Pi_k)+O(k^4), $$
so that
$$\hat A^{-1}(k)=\frac8{9k^2}(3\mathds 1-2\Pi_k)(1+O(k^2)).$$
Using these properties of $\hat A^{-1}(k)$ we see that the function $F(k)$ in \eqref{eF} is even, uniformly bounded on $\mathcal B$, and analytic in $k$ away from $k=0$.
In the vicinity of the singularity, it behaves like
\begin{equation} F(k)=\frac2{k^4}(k_1^4+k_2^4-\scalebox{1.2}{$\frac23$} k_1^2k_2^2)+O(k^2)=2-\frac{16}{3}\frac{k_1^2k_2^2}{k^4}+O(k^2).\label{Fexp}\end{equation}
In order to extract the dominant contributions from \eqref{FF}, we rewrite $F(k)=F((0,k_2))+[F(k)-F((0,k_2))]$, with $F((0,k_2))=2$. The contribution from $F((0,k_2))$ reads, for any small $\epsilon>0$:
$$\frac12\int_{\mathcal B} \frac{dk}{|\mathcal B|}\, \frac{1-\cos(k_1n)}{1-\cos k_1}=\frac1{2\pi}\int_{-\epsilon}^\epsilon d k_1 \frac{1-\cos (k_1n)}{k_1^2}+O(1)=\frac{n}{2}+O(1),$$
where the remainder $O(1)$ is uniformly bounded as $n\to \infty$. By using \eqref{Fexp}, we can rewrite
the contribution from $[F(k)-F((0,k_2))]$,
for any small $\epsilon>0$, as
\begin{equation}\label{log}\begin{split}&\frac1{4}\int_{[-\epsilon,\epsilon]^2} \frac{dk}{|\mathcal B|}\ \frac{1-\cos(k_1n)}{k_1^2/2}\Big(-\frac{16}3\frac{k_1^2k_2^2}{k^4}\Big)+O(1)=\\
&=
-\frac1{\sqrt3\,\pi^2}\int_{[-\epsilon,\epsilon]^2} dk\ \frac{k_2^2(1-\cos(k_1n))}{k^4}+O(1),\end{split}
\end{equation}
where, again, the remainder $O(1)$ is uniformly bounded as $n\to \infty$. The dominant term in \eqref{log} can be computed explicitly, and gives
\begin{equation}\eqref{log}=-\frac1{\sqrt{3}\,\pi}\int_{0}^{\epsilon}\frac{dk_1}{k_1}(1-\cos(k_1n))+O(1)=-\frac1{\sqrt3\,\pi}\log n+O(1).\end{equation}
Putting things together, we obtain \eqref{dip.0}, as desired.

\subsection{The energy of a pair of infinite, parallel, grain boundaries (proof of \eqref{RS.0}).}
Let $$\sigma^{M,n,m}_\mathrm{grain}(x,x')=\sum_{j=0}^{M-1}\sigma^n_\mathrm{dip}(x-mj(b_2-b_3),x'-mj(b_2-b_3)),$$ with $\sigma^n_\mathrm{dip}$ defined in~\eqref{defsigdip.eq}.
By proceeding as in the previous subsection, we find that
$$\frac1{M}\left\langle\sigma^{M,n,m}_\mathrm{grain}, B\sigma^{M,n,m}_\mathrm{grain}\right\rangle=\frac{n}{2},$$ while
\begin{equation}\label{eM}\begin{split}&
\frac1{M}\left\langle\dd^*B\sigma^{M,n,m}_\mathrm{grain}, A^{-1} \dd^* B\sigma^{M,n,m}_\mathrm{grain}\right\rangle=\\
&=\frac1{4M}\int_{\mathcal B} \frac{dk}{|\mathcal B|}\, \frac{1-\cos(k_1n)}{1-\cos k_1}\frac{1-\cos(Mm\sqrt3\, k_2)}{1-\cos(m\sqrt3\, k_2)}\,F(k),\end{split}\end{equation}
where $F(k)$ is the same as in \eqref{eF}. If we now let $M\to\infty$,
\begin{equation}\begin{split} &\lim_{M\to\infty}\frac1{M}\left\langle \dd^* B\sigma^{M,n,m}_\mathrm{grain}, A^{-1} \dd^* B\sigma^{M,n,m}_\mathrm{grain}\right\rangle=\\
&=\frac1{16\pi m}\sum_{j=0}^{2m-1}\int_0^{2\pi}dk_1\, \frac{1-\cos(k_1n)}{1-\cos k_1}F\big((k_1,p_{j})\big),
\end{split}\end{equation}
where $p_{j}=p_j(m)=\frac{2\pi j}{m\sqrt3}$.
In order to compute this expression asymptotically, as $n,m\to\infty$, it is convenient to rewrite $F((k_1,p_{j}))$
as $[F((k_1,p_j))-F((0,p_j))]+F((0,\frac{2\pi j}{m\sqrt3}))$, where $F((0,p_j))=2$ (in the case $j=0$, this identity should be understood as $\lim_{k_1\to 0}$ $F((k_1,0))=2$). The contribution from $F((0,p_j))$
reads:
$$\frac1{8\pi m}\sum_{j=0}^{2m-1}\int_0^{2\pi}dk_1\, \frac{1-\cos(k_1n)}{1-\cos k_1}=\frac{n}2,$$
while the one  from $[F((k_1,p_j))-F((0,p_j))]$ reads:
$$\frac1{16\pi m}\sum_{j=0}^{2m-1}\int_0^{2\pi}dk_1\, \frac{1-\cos(k_1n)}{1-\cos k_1}\big[F((k_1,p_j))-2\big].$$
A computation shows that the difference in brackets is $O(k_1^2)$ for $k_1$ close to $0$ (possibly non-uniformly in $j,m$); correspondingly,
if we let $n\to\infty$, the term proportional to $\cos(n k_1)$ under the integral sign goes to zero
as $(\log n)/n$. Summarizing,
$$\lim_{n\to\infty}E_\mathrm{grain}(n,m)=\frac1{32\sqrt3\pi m^2}\sum_{j=0}^{2m-1}\int_0^{2\pi}dk_1\, \frac{2-F((k_1,p_j))}{1-\cos k_1}+O(1/m).$$
The dominant contribution to the first term in the right side as $m\to\infty$ comes from the region $(k_1,p_j)\in[-\epsilon,\epsilon]^2$, the contribution from the complement being bounded uniformly in $m$
(here $\epsilon$ is an arbitrary small, positive, constant). Moreover, by rewriting $1-\cos k_1$ for $k_1$ small as $\frac{k_1^2}2(1+O(k_1^2))$, and by expanding $F(k)$ as in \eqref{Fexp},
we find, letting $J_\epsilon=\lfloor \sqrt3 \epsilon m/(2\pi)\rfloor$,
$$\lim_{n\to\infty}E_\mathrm{grain}(n,m)=\frac4{3\sqrt3\pi m^2}\sum_{j=0}^{J_\epsilon}\int_{0}^{\epsilon}dk_1\, \frac{p_j^2}{(k_1^2+p_j^2)^2}+O(1/m).$$
Finally, recalling that $p_{j}=\frac{2\pi j}{m\sqrt3}$, by summing over $j$ and integrating over $k_1$, we find:
\begin{equation}\lim_{n\to\infty}E_\mathrm{grain}(n,m)=\frac{\log m}{6\pi m} + O(1/m),\label{RS}\end{equation}
asymptotically as $m\to\infty$. This is the desired `Read-Shockley' law for the energy of a grain boundary.

\subsubsection{Comparison of the Read-Shockley formula with the capacitor law.}\label{sec.compar}
As promised above, let us now make a technical comparison between the derivation of the Read-Shockley formula \eqref{RS} and the analogous computation in the case that the operator $B$ is replaced by the identity.
In this case we lose the key feature of our discrete elasticity model, that is, invariance under linearized rotations. This is the physical reason why the scaling of the
corresponding energies are completely different.

More specifically, let $u$ be a $\mathbb R^2$-valued function on $\mathcal T$, $\sigma$ a lattice-valued function on the nearest neighbor bonds of $\mathcal T$, and $q$ a lattice-valued function on the faces of $\mathcal T$,
with finite support and zero total charge, $\sum_f q(f)=0$. Consider the minimum energy defined by
\begin{eqnarray*}
\mathcal E = \frac12\min_{(u,\sigma)} \left\{ |\dd u - \sigma|^2\; : \; \dd \sigma=q \right\} = \frac12\min_v \left\{ |v|^2\; : \; \dd v=q \right\},
\end{eqnarray*}
where the minimum over $v$ is performed over $\mathbb R^2$-valued (rather than $\mathcal T$-valued) functions on the nearest neighbor bonds of $\mathcal T$.
The minimizer $\sigma_q$ is characterized by the Euler-Lagrange equations $\dd \sigma_q=q$ and $\dd^* \sigma_q =0$. Clearly $\sigma_q = \dd^* \Delta^{-1} q$ satisfies those equations and is the unique minimizer. Hence
\begin{eqnarray*}
\mathcal E = \frac{1}{2} \langle q,\Delta^{-1} q\rangle,
\end{eqnarray*}
i.e. the modified Ariza-Ortiz energy reduces to a lattice Coulomb interaction. To compute $\mathcal E$ for specific 2-forms $q$ it is convenient to work with the Fourier representation of the Laplacian acting on 2-forms.
A simple calculation shows that
$$ \Delta q(x,j) = \begin{cases} 3q(x,1) - q(x,2)- q(x+b_2,2)-q(x-b_3,2) & \text{ if } j = 1,\\
3q(x,2) - q(x,1)-q(x-b_2,1)- q(x+b_3,1) & \text{ if } j = 2,
\end{cases}$$
where we use the abbreviation $q(x,1)=q(x,x+b_1,x-b_3)$ and $q(x,2) = q(x,x-b_2,x+b_1)$, see Fig.~\ref{fig.delta} for a visualization of the two face types.

\begin{figure}\begin{center}
\includegraphics[page=1]{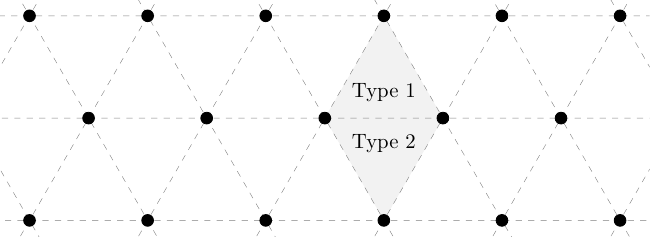}\end{center}
\caption{The face types of the triangular lattice.}
\label{fig.delta}
\end{figure}
We write the Fourier transform of 2-forms $q$ by
\begin{equation*}
q(x,j)=\int_{\mathcal B} \frac{dk}{|\mathcal B|}\hat q(k,j) e^{-ikx},\, x\in\mathcal T,\quad {\rm with}\quad j=1,2
\end{equation*}
and obtain the Fourier symbol
$$\widehat{\Delta q}(k)=\begin{pmatrix} 3 & -\Omega(k) \\ -\Omega^*(k) & 3\end{pmatrix} \hat q(k),$$
where $\Omega(k)=1+e^{-ik\cdot b_2}+e^{ik\cdot b_3}$ and we used the convention that each coefficient of the symbol is interpreted as a multiple of the 2-dimensional identity matrix (i.e. the symbol is actually a Hermitian $4\times 4$-matrix).
In conclusion,
$$\mathcal E=\frac12(q,\Delta^{-1}q)=\frac12\int_\mathcal B\frac{dk}{|\mathcal B|} \frac1{9-|\Omega(k)|^2}\hat q^T(-k) \begin{pmatrix} 3 & \Omega(k) \\ \Omega^*(k) & 3\end{pmatrix} \hat q(k).$$
Let us now compute the energy of the dislocation dipole: If $q=q_\mathrm{dip}^n$ (cf.~\eqref{disl_dist}) then the corresponding energy is:
$$\mathcal E_\mathrm{dip}(n)=3\int_{\mathcal B}\frac{dk}{|\mathcal B|} \frac{1-\cos (k_1n)}{9-|\Omega(k)|^2}.$$
Note that, close to the singularity $k=0$, we have $|\Omega(k)|^2=9-\frac{3}{2}|k|^2+O(k^3)$, so that, for any $\epsilon>0$,
$$\mathcal E_\mathrm{dip}(n)=\frac2{|\mathcal B|}\int_{[-\epsilon,\epsilon]^2}\hskip-.2truecm \frac{1-\cos (k_1n)}{|k|^2} \, dk+O(1)=\frac{\sqrt3}{2\pi}\log n+O(1),$$
which is qualitatively the same as the energy of the dislocation dipole, \eqref{dip.0}. Since the energy of a single dipole is asymptotically the same at large distances, up to a multiplicative constant,
both for this lattice Coulomb case and the standard case of the Ariza-Ortiz model, one may naively expect that the energy of two parallel arrays of dipoles is also qualitatively the same in the two models. However, this is not the
case. If we consider a charge distribution which resembles two parallel capacitor plates where $q = q^{M,n,m}_\mathrm{grain}$ (cf. \eqref{wall_dist}), then
\begin{eqnarray*}
\mathcal E_\mathrm{grain}(n,m)&=&\lim_{M\to \infty} \frac{\sqrt3}{mM}\int_{\mathcal B}\frac{dk}{|\mathcal B|} \frac{1-\cos (k_1n)}{9-|\Omega(k)|^2}\frac{1-\cos(mM\sqrt3 k_2)}{1-\cos(m\sqrt3 k_2)}\\
&=&\lim_{M\to\infty}\frac{\sqrt3}{4\pi m^2}\sum_{j=0}^{2m-1}\int_{0}^{2\pi} dk_1 \frac{1-\cos (k_1n)}{9-|\Omega(k_1,p_j)|^2},
\end{eqnarray*}
where $p_{j}=p_j(m)=\frac{2\pi j}{m\sqrt3}$. The dominant contribution to the right side as $n\to\infty$ at $m$ fixed comes from the region
$(k_1,p_j)\in [-\epsilon/m,\epsilon/m]^2$ mod $\mathcal L^*$,
for any small $\epsilon>0$, the contribution from the complement being bounded
from above uniformly in $n$, as $n\to\infty$ (this is an immediate consequence of the fact that the only zero of $9-|\Omega(k)|^2$ is in $k=0$).
On the other hand, the contribution from $(k_1,p_j)\in [-\epsilon/m,\epsilon/m]^2$ mod $\mathcal L^*$ grows linearly in $n$, as $n\to\infty$, so that,
noting that the $9-|\Omega(k_1,0)|^2=\frac32 k_1^2+O(k_1^3)$ for $k_1$ small,
$$\lim_{n\to\infty}\frac1{n}\mathcal E_\mathrm{grain}(n,m)=\lim_{n\to\infty}\frac1{n}\frac{\sqrt3}{4\pi m^2}\int_{-\epsilon/m}^{\epsilon/m} dk_1 \frac{1-\cos (k_1n)}{\frac32k_1^2}=\frac{\sqrt3}{2m^2}, $$
which is the usual electrostatic energy of an infinite capacitor. Note the linear behaviour of the energy in $n$, as $n\to\infty$, to be compared with the asymptotic {\it independence} of the
energy of a grain boundary in $n$ in the Ariza-Ortiz model, see \eqref{RS.0}.

\appendix

\section{Interpretation of the model}\label{app:interpretation}

In this appendix we provide a heuristic interpretation of the Ariza-Ortiz model, which aims at clarifying its connection with more fundamental microscopic models for atomic crystals.
For a more systematic discussion, from a different perspective, the reader is referred to the original paper \cite{ao05} where the model has been introduced.

For simplicity, we restrict our discussion to two-dimensions. Assume that the particles interact via a classical, rotationally invariant, pair potential $V$, with a non-degenerate minimum at distance,
say, $1$. If such a minimum is deep and narrow, the potential energy of a particle configuration can be well approximated by a sum over nearest neighbors:
\begin{equation}\label{eq:DT}E(z) = \frac12\sum_{\substack{\xi,\eta\in DT(z):\\ \xi\sim \eta}} V(|z(\xi)-z(\eta)|),\end{equation}
where the sum runs over ordered pairs of nearest neighbor sites of the graph $DT(z)$, the Delaunay triangulation of the particle configuration $z$ (i.e., the dual of the Voronoi diagram of $z$);
in \eqref{eq:DT}, $z(\xi)$ and $z(\eta)$ indicate the coordinates in $\mathbb R^2$ of the vertices of the graph $DT(z)$ labelled $\xi$ and $\eta$, respectively. Under the assumption that the pair potential $V$ has a
deep, narrow, minimum, located at $1$, we expect that the low-energy particle configurations are such that the nearest neighbor pairs involved in the summation in \eqref{eq:DT} have distance $|z(\xi)-z(\eta)|$
close to $1$. See Fig.~\ref{fig:DT} for an example.

\begin{figure}[ht]\begin{center}
\includegraphics[scale=0.33]{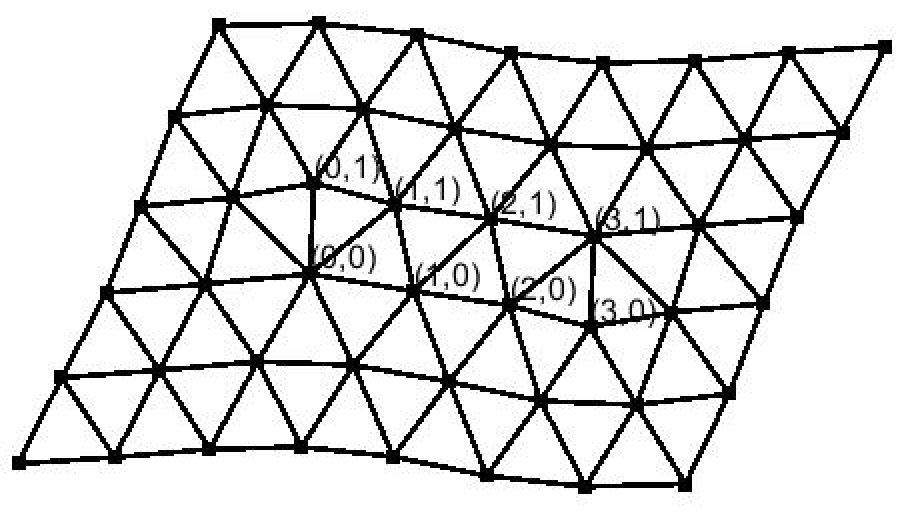}\end{center}
\caption{A particle configuration $z$ and its Delaunay triangulation $DT(z)$. Notice that, in this example, the graph $DT(z)$ is not a global deformation of the regular triangular lattice:
there are sites with 7 neighbours ((0,0) and (3,1)) next to sites with 5 neighbours ((0,1) and (3,0)).
These kinds of defects correspond to dislocations, which play a key role in our analysis.}\label{fig:DT}
\end{figure}

The inconvenient feature of \eqref{eq:DT} is that the sum runs over the nearest neighbor sites of a graph, whose structure depends upon the configuration itself. A more convenient way of expressing the
same energy is to reduce to a fixed reference graph, after appropriate redefinition of the nearest neighbor bonds. In our 2D setting, the natural reference graph is the regular triangular lattice of unit mesh, denoted
$\mathcal T$. Given a particle configuration\footnote{Implicitly, we assume that the particle configuration $z$ is sufficiently `reasonable' that the desired one-to-one correspondences between vertices and nearest-neighbor
pairs of $DT(z)$ and $\mathcal T$ are well-defined. We expect that particle configurations are almost-all reasonable, with respect to the infinite volume Gibbs measure with energy \eqref{eq:DT} and inverse temperature $\beta$, provided $\beta$ is sufficiently large.}
$z$, we establish a one-to-one correspondence of the vertices of $DT(z)$ with those of $\mathcal T$ (call it the `vertex correspondence') and a one-to-one correspondence of the ordered nearest neighbor pairs of $DT(z)$ with those of $\mathcal T$ (call it the `bond correspondence'), in such a way that
\eqref{eq:DT} is re-expressed as a sum over ordered pairs of nearest neighbor sites of $\mathcal T$ (note that the bond correspondence we introduce is {\it not} the one induced by the vertex correspondence, see below for its definition and an explicit example):
\begin{equation}E(z)=\frac12\sum_{\substack{x,y\in\mathcal T:\\ x\sim y}} V(|z(\varphi(x,y))-z(\psi(x,y))|).\label{eq:sumT}\end{equation}
Here $(\varphi(x,y),\psi(x,y))$ is the image of an ordered nearest neighbor pair in $DT(z)$ under the aforementioned bond correspondence\footnote{Reversing the order of the nearest neighbor pair $(x,y)$,
we have $(\varphi(x,y),\psi(x,y))=(\psi(y,x),\varphi(y,x))$, which is a constraint to be imposed on the functions $\varphi,\psi$.}. Letting
\begin{eqnarray}
\label{def-disp}
u(x) &=& z(x)-x\\
\label{def-slip}
\sigma(x,y) &=& y-x -\psi(x,y)+\varphi(x,y)
\end{eqnarray}
where $u$ and $\sigma$ are the displacement and slip fields, respectively,
eq.\eqref{eq:sumT} can be further rewritten as
\begin{equation}
E(z) =\frac12\sum_{\substack{x,y\in\mathcal T:\\x\sim y}}V(|x-y+u(\varphi(x,y))-u(\psi(x,y))+\sigma(x,y)|).\label{preAO}
\end{equation}
The slip-field $\sigma$ is a key feature of our model, and it provides a direct way of measuring the change of the nearest neighbors, as well as the occurence of
dislocations. For example, if the `charge' $q = \sum_{i} \sigma(x_i,x_{i+1})$ is different from zero for some closed path $x_i\in \mathcal T$, then the path encloses a dislocation defect.
The Burger's vectors are closely related to the charges $q$ associated with elementary circuits (the boundaries of the triangular faces of $\mathcal T$), but also
depend on the orientation and the location of the path;
their specific definition is unimportant for our purposes and, therefore, we skip it.

\medskip
{\it Example.} Let us illustrate in a concrete case how to construct the
one-to-one mappings between the vertices and nearest neighbor pairs of $DT(z)$ and those of $\mathcal T$, and how to compute the slip field.
Consider the graph $DT(z)$ associated with the configuration of Fig.~\ref{fig:DT}, and deform it so that its vertex set coincides with that of $\mathcal T$, as depicted in Fig.~\ref{fig:T}.

\begin{figure}[ht]\begin{center}
\includegraphics[scale=0.33]{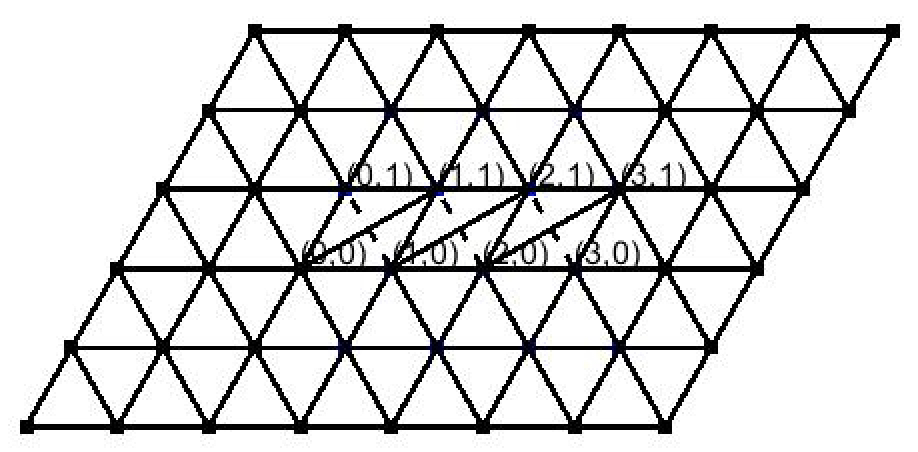}\end{center}
\caption{A deformation of the graph in Fig.~\ref{fig:DT} with the same vertex set as the regular triangular lattice. The solid bonds correspond to the nearest neighbor edges
 of the original graph in Fig.~\ref{fig:DT}. The dashed bonds are the `missing' nearest neighbor edges of $\mathcal T$.}
 \label{fig:T}
\end{figure}

In Figures \ref{fig:DT} and \ref{fig:T} we used the same labels for the corresponding vertices under the mapping induced by the deformation of the graph. Moreover, in Fig.~\ref{fig:T}, we
connected by solid lines the
images of the nearest neighbor pairs of $DT(z)$. These solid lines can be put in one-to-one correspondence with the nearest neighbor pairs of $\mathcal T$; specifically, for any positively oriented\footnote{
We use the convention that the `positively oriented' nearest neighbor pairs of $\mathcal T$ are those of the form $(x,x+b_l)$, with $l=1,2,3$ (here $b_1=\left({1\atop 0}\right)$, $b_2=\left({-1/2\atop \sqrt3/2}\right)$, and $b_3=-b_1-b_2$); of course, the nearest neighbor pairs of the form $(x,x-b_l)$ are called negatively oriented. It is sufficient to define the functions $\varphi,\psi$ and the slip field on positively oriented
n.n. pairs: if $(x,y)$ is negatively oriented, we let $(\varphi(x,y),\psi(x,y))=(\psi(y,x),\varphi(y,x))$, so that $\sigma(x,y)=-\sigma(y,x)$.} nearest
neighbor bond $(x,y)$ of $\mathcal T$, we let $\varphi(x,y)\equiv x$, and
$$\psi(x,y)=\begin{cases}
(0,0) & \text{if $(x,y)=((1,1), (1,0))$,}\\
(1,1) & \text{if $(x,y)=((1,0), (0,1))$,}\\
(1,0) & \text{if $(x,y)=((2,1),(2,0)))$,}\\
(2,1) & \text{if $(x,y)=((2,0),(1,1))$,}\\
(2,0) & \text{if $(x,y)=((3,1),(3,0))$,}\\
(3,1) & \text{if $(x,y)=((3,0), (2,1))$,}\\
y & \text{otherwise.}\end{cases}$$
With these conventions, the slip field on positively oriented edges is
\begin{equation}\label{slipex}\sigma(x,y)=\begin{cases}-b_1 & \text{if $(x,y)\in \{((1,0),(0,1)),\, ((2,0),(1,1)),\,((3,0),(2,1))\}$,}\\
b_1 & \text{if $(x,y)\in\{((1,1),(1,0)),\,((2,1),(2,0)),\,((3,1),(3,0))\}$,}\\
0 & \text{otherwise.}\end{cases}\end{equation}

\medskip

The Ariza-Ortiz model is obtained from eq.\eqref{preAO} under a couple additional approximations. Letting $x-y\equiv \ell_0$ and $u(\varphi(x,y))-u(\psi(x,y))+\sigma(x,y)\equiv \delta_0$,
we rewrite $V(|\ell_0+\delta_0|)$ by expanding it in Taylor series around the minimum: recalling that $|\ell_0|=1$ and assuming $\delta_0$ to be small, we find:
$$V(|\ell_0 +\delta_0|) = V(1)+\frac{V''(1)}2 (\ell_0\cdot\delta_0)^2+ O\(|\delta_0|^3\).$$
Neglecting the remainder, dropping an additive constant and rescaling the resulting energy, we obtain:
\begin{equation}E(z)\approx \frac14\sum_{\substack{x,y\in\mathcal T:\\x\sim y}}\big[(x-y)\cdot\big(u(\varphi(x,y))-u(\psi(x,y))+\sigma(x,y)\big)\big]^2.\label{preAObis}
\end{equation}
This quadratic approximation corresponds to the standard small-strain assumption in continuum mechanics. Finally, we replace
\begin{equation}
u(\varphi(x,y))-u(\psi(x,y)) \approx u(x)-u(y),
\end{equation}
which corresponds to the `linearized plasticity' approximation in continuum mechanics. After these replacements, we obtain
\begin{equation}E(z)\approx \frac14\sum_{\substack{x,y\in\mathcal T:\\x\sim y}}\big[(x-y)\cdot\big(u(x)-u(y)+\sigma(x,y)\big)\big]^2,
\end{equation}
which is the Ariza-Ortiz model. 

All the approximations involved in the previous scheme are uncontrolled, and their validity should be (at least) checked a posteriori, by showing that the `typical', low energy, configurations
of the Ariza-Ortiz Hamiltonian are close (in a sense to be defined) to those of the original, realistic, Hamiltonian. This remains to be done: in fact, proving (even at heuristic level) the correctness of
these approximations is a major challenge in the field and goes beyond the purposes of our paper. From a numerical point of view, the use of the Ariza-Ortiz Hamiltonian gives results in qualitative
agreement with the more realistic energy function \eqref{preAO}, see Fig.~\ref{dislocation_pair.fig}.

\begin{figure}[h]
\centering
\begin{subfigure}{.5\textwidth}
\includegraphics[width=1.1\linewidth]{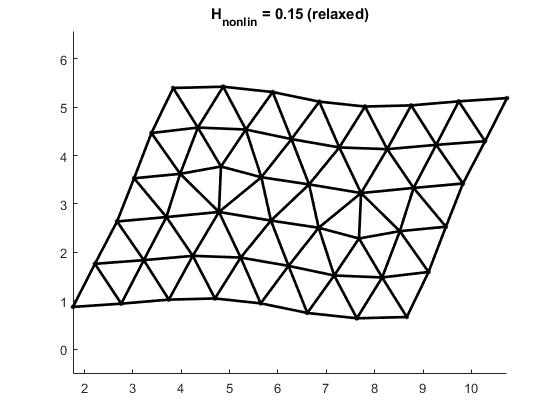}
\end{subfigure}%
\begin{subfigure}{.5\textwidth}
\centering
\includegraphics[width=1.1\linewidth]{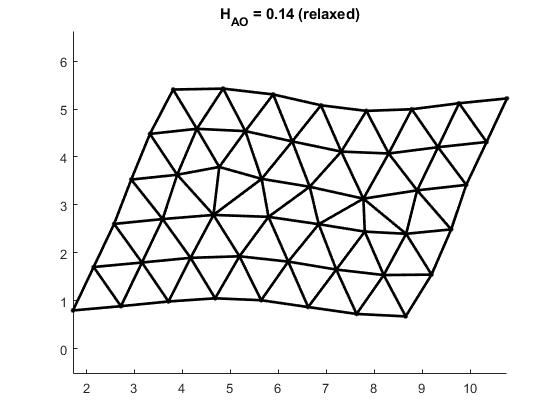}
\end{subfigure}
\caption{A visualization of the relaxed configuration corresponding to the minimum over $u$ of: the energy function in \eqref{preAO} with $V(x)=x^2/2$ at fixed Delaunay graph, equal to the one of Fig.~\ref{fig:DT}
(left panel); the Ariza-Ortiz Hamiltonian at fixed slip field, equal to the one in \eqref{slipex}.}
\label{dislocation_pair.fig}
\end{figure}

\section{On the energy of a grain supported on an infinite vertical strip}\label{app.grain}

In this appendix we discuss the connection between the energy $E_{\text{grain}}(n,m)$ defined in \eqref{defEgrain} and the optimal grain energy $\mathcal E_{\mathcal G}(S)$ defined in
\eqref{eq:GSen}, in the case of a grain $\mathcal G$ supported in a vertical strip of width $n$, `rotated' by an `angle' $\theta\sim 1/m$. As anticipated in Sect.\ref{subsec:RS}, for simplicity,
we restrict our attention to two-dimensions. We recall that $\mathcal T$ is the infinite 2D triangular lattice with basis vectors $b_1=\left({1\atop 0}\right)$, $b_2=\left({-1/2\atop \sqrt3/2}\right)$, and we let
$b_3=-b_1-b_2$. We also let $m_1=\frac{4\pi}{\sqrt{3}}\left({\sqrt{3}/2 \atop\phantom{-}1/2}\right)$,
$m_2= \frac{4\pi}{\sqrt{3}}\left({0 \atop1}\right)$ be a basis of the dual lattice $\mathcal T^*$, such that $b_i\cdot m_j=2\pi \delta_{i,j}$, for $i,j=1,2$; moreover,
we define $m_3=m_1-m_2=\frac{4\pi}{\sqrt{3}}\left({\sqrt{3}/2\atop-1/2}\right)$, so that $b_3\cdot m_3=0$. 

We consider a grain whose support is an infinite vertical strip of width $n$:
\begin{equation}\label{J.-2}\mathcal G=\mathcal G_a\cup \mathcal G_b\quad \text{with} \quad \mathcal G_a:=\{x=n_1b_1+n_2(b_2-b_3) : 0\le n_1\le n, n_2\in\mathbb Z\},
\quad \mathcal G_b:=\mathcal G_a+b_2.\end{equation}

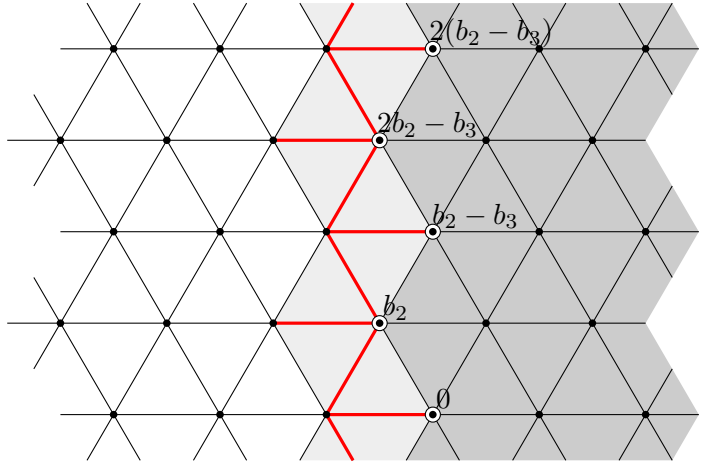
\begin{figure}[!h]
	\centering
	  \begin{tikzpicture}[scale=0.7]

\fill[gray!40] (-0.5,-0.5*1.73205) -- (0,0) -- (-1,1.73205) -- ++ (1,1.73205) -- ++ (-1,1.73205) -- ++ (1,1.73205) -- ++ (-0.5,0.5*1.73205) -- ++ (5,0) --
++ (0.5,-0.5*1.73205) -- ++ (-1,-1.73205) -- ++ (1,-1.73205) -- ++ (-1,-1.73205) -- ++ (1,-1.73205) -- ++ (-0.5,-0.5*1.73205) -- cycle;

\fill[gray!13] (-0.5,-0.5*1.73205) -- (0,0) -- (-1,1.73205) -- ++ (1,1.73205) -- ++ (-1,1.73205) -- ++ (1,1.73205) -- ++ (-0.5,0.5*1.73205) -- ++ (-2,0) --
++ (0.5,-0.5*1.73205) -- ++ (-1,-1.73205) -- ++ (1,-1.73205) -- ++ (-1,-1.73205) -- ++ (1,-1.73205) -- ++ (-0.5,-0.5*1.73205) -- cycle;

\draw (-7,0)--(5,0);
\draw (-7,2*1.73205)--(5,2*1.73205);
\draw (-7,4*1.73205)--(5,4*1.73205);
\draw (-8,1*1.73205)--(4,1*1.73205);
\draw (-8,3*1.73205)--(4,3*1.73205);
 \foreach \x in {0,...,3}
{
\draw (-6.5+2*\x,-0.5*1.73205)--(-1.5+2*\x,4.5*1.73205);
\draw (-6.5+2*\x,4.5*1.73205)--(-1.5+2*\x,-0.5*1.73205);
}

\draw (-7.5,0.5*1.73205) -- (-3.5,4.5*1.73205);
\draw (-7.5,2.5*1.73205) -- (-5.5,4.5*1.73205);
\draw (1.5,-0.5*1.73205) -- (4.5,2.5*1.73205);
\draw (3.5,-0.5*1.73205) -- ++ (1,1.73205);
\draw (-5.5,-0.5*1.73205) -- ++ (-2,2*1.73205);
\draw (-3.5,-0.5*1.73205) -- ++ (-4,4*1.73205);
\draw (1.5,4.5*1.73205) -- ++ (3,-3*1.73205);
\draw (3.5,4.5*1.73205) -- ++ (1,-1.73205);

\draw[red, very thick] (0,0) -- (-2,0);
\draw[red, very thick] (-1,1.73205) -- (-2,0);
\draw[red, very thick] (-1,1.73205) -- ++ (-2,0);
\draw[red, very thick] (-1,1.73205) -- ++ (-1,1.73205);
\draw[red, very thick] (0,2*1.73205) -- ++ (-2,0);
\draw[red, very thick] (-1,3*1.73205) -- ++ (-1,-1.73205);
\draw[red, very thick] (-1,3*1.73205) -- ++ (-2,0);
\draw[red, very thick] (-1,3*1.73205) -- ++ (-1,1.73205);
\draw[red, very thick] (0,4*1.73205) -- ++ (-2,0);
\draw[red, very thick] (-2,0) -- ++ (0.5,-0.5*1.73205);
\draw[red, very thick] (-2,4*1.73205) -- ++ (0.5,0.5*1.73205);

\draw (0,0) node[vertex,E] {};
\draw (0,2*1.73205) node[vertex,E] {};
\draw (0,4*1.73205) node[vertex,E] {};
\draw (-1,1.73205) node[vertex,E] {};
\draw (-1,3*1.73205) node[vertex,E] {};
\foreach \x in {-3,...,2}
{
\draw (2*\x,0) node[smallvertex] {};
\draw (2*\x,2*1.73205) node[smallvertex] {};
\draw (2*\x,4*1.73205) node[smallvertex] {};
\draw (-1+2*\x,1.73205) node[smallvertex] {};
\draw (-1+2*\x,3*1.73205) node[smallvertex] {};}

\draw (0.2,0.3) node {$0$};
\draw (-0.7,1.73205+0.3) node {$b_2$};
\draw (0.8,2*1.73205+0.3) node {$b_2-b_3$};
\draw (-0.1,3*1.73205+0.3) node {$2b_2-b_3$};
\draw (1.1,4*1.73205+0.3) node {$2(b_2-b_3)$};
\end{tikzpicture}
\caption{A portion of the grain: the dark gray area represents the grain, and the circled sites are those belonging to its left boundary. The light gray area is the left `boundary layer', and the
red bonds (the `boundary bonds') are those connecting sites in the grain with sites in its complement: those are bonds on which the slip field $\sigma$ in the minimization problem Eq.\eqref{mintobed}
may be non zero; correspondingly, the light gray faces are those where the `charge' $d\sigma$ may be non zero.}\label{fig:1}
\end{figure}

Note that the complement of the grain is disconnected and consists of two semi-infinite components, the left one, denoted $\mathcal G^c_L$, and the right one, denoted $\mathcal G^c_R$.
We let
\begin{equation}\label{J.-1}u_S(x)=\begin{cases} 0 & \text{if $x\in\mathcal G^c_L$}\\
Sx+\tau & \text{if $x\in\mathcal G$}\\
\tau_R & \text{if $x\in\mathcal G^c_R$}\end{cases}\end{equation}
with $S=S(\theta)=\sqrt{3}\,\theta\begin{pmatrix} 0 & 1 \\ -1 & 0\end{pmatrix}$ for some $\theta\in\mathbb R$
(the normalization factor $\sqrt3$ is chosen for later convenience, e.g.,
for having a prefactor $\frac{\theta}{2\pi}$, rather than $\frac{\theta}{2\pi\sqrt3}$, in eq.\eqref{J.1})
and $\tau,\tau_R$ to be fixed. We are interested in estimating the optimal grain energy per unit vertical length that, in analogy with the definition \eqref{eq:GSen} for the finite-grain case, is defined as follows:
\begin{equation}\label{J.-3}e_{\mathcal G}(S)=\liminf_{\Lambda\to\mathcal T}\frac1{\tfrac{\sqrt3}2N}\lim_{\epsilon\to 0^+}\inf_{\sigma\in \mathcal M_S^{(\epsilon)}(\mathcal G)}\inf_u H_{\text{AO}}(u,\sigma),\end{equation}
where we recall that $\Lambda=\Lambda^{(N)}$ is the 2D analogue of \eqref{LN}, whose vertical height is $\tfrac{\sqrt3}2N$ (which explains the normalization factor $1/({\tfrac{\sqrt3}2N})$ in the right side
of \eqref{J.-3}), and
$\mathcal M_S^{(\epsilon)}(\mathcal G)$ is the set of lattice-valued slip fields $\sigma$ such that there exist $\tau,\tau_R$ for which $\sigma,\tau,\tau_R$ realize the infimum of $\inf_{\tau,\tau_R}\inf_\sigma^*
H_{\text{AO}}(u_S,\sigma)$  within a precision $\epsilon$, where the $*$ on $\inf^*_\sigma$ indicates the constraint that the support of $\sigma$ is over the bonds connecting the grain with its complement, such as the red bonds of Fig.\ref{fig:1}.

Let us focus on the minimization problem defining the set $\mathcal M_S^{(\epsilon)}(\mathcal G)$, i.e.,
\begin{equation}\label{mintobed}\begin{split}&\big({\tfrac{\sqrt3}2N}\big)^{-1}\inf_{\tau,\tau_R}\inf_\sigma^*H_{\text{AO}}(u_S,\sigma)=\\
&=\big({\tfrac{\sqrt3}2N}\big)^{-1}\inf_{\tau,\tau_R}\inf_\sigma^*\frac12
\sum_{x\sim y}\big[(x-y)\cdot(u_S(x)-u_S(y)+\sigma(x,y))\big]^2.\end{split}\end{equation}
Note that by definition the only non-zero terms in the sum in the right side are those associated with boundary bonds $(x,y)$, with $x\in\mathcal G$ and $y\not\in\mathcal G$ (such as the red bonds of Fig.\ref{fig:1}): for such bonds
$u_{S}(x)-u_{S}(y)=Sx+\tau$, if $x$ belong to the left boundary, and $u_{S}(x)-u_{S}(y)=Sx+\tau-\tau_R$, if $x$ belong to the right boundary. The goal is to find a minimizer $\sigma(x,y)$ for $(x,y)$ a boundary bond. For this purpose, it is convenient to decompose $S$ into simple slip systems.
Recall that any $2\times2$ skew-symmetric matrix $A$ can be decomposed into simple slip systems, i.e., $
A=\sum_{l=1}^3 \xi_l \,b_l\otimes m_{n(l)}$,
for suitable coefficients $\xi_l$, where $m_{n_l}\in\mathcal T^*$ are the slip plane normals, namely: $m_{n(1)}=m_2$, $m_{n(2)}=m_1$, and $m_{n(3)}=m_3$. In particular, a simple computation shows that
\begin{equation}\label{J.1}S=S(\theta)=\frac\theta{2\pi}\big(b_1\otimes m_2-b_2\otimes m_1+b_3\otimes m_3\big).\end{equation}
Setting temporarily $\tau=0$,
Eq.\eqref{J.1} suggests the following choice for the slip-field minimizer for $(x,y)$ a boundary bond with $x\in\mathcal G$ and $y\not\in\mathcal G$:
\begin{equation}\sigma(x,y)= -b_1\langle \tfrac{\theta}{2\pi} x\cdot m_2\rangle+b_2 \langle \tfrac{\theta}{2\pi} x\cdot m_1\rangle-b_3 \langle \tfrac{\theta}{2\pi} x\cdot m_3\rangle+\tau_R \mathds 1_{y\in\mathcal G^c_R},\end{equation}
where $\langle x\rangle:=\lfloor x+\tfrac12\rfloor$ denotes the `nearest integer function'. On the left boundary (similar considerations hold for the right one), $x$ equals $n_2(b_2-b_3)$ or
$n_2(b_2-b_3)+b_2$ for some $n_2\in\mathbb Z$, depending on whether $x$ is in $\mathcal G_a$ or $\mathcal G_b$. Note that, for $x=n_2(b_2-b_3)$, we have:
$\tfrac1{2\pi}x\cdot m_2=2n_2$, $\tfrac1{2\pi}x\cdot m_1=n_2$,
and $\tfrac1{2\pi}x\cdot m_3=-n_2$. Therefore, on the left boundary, if $x$ equals $n_2(b_2-b_3)$ or $n_2(b_2-b_3)+b_2$, recalling that $b_2+b_2=-b_1$,
\begin{equation}\sigma(x,y)= -b_1(\langle 2\theta n_2\rangle+\langle \theta n_2\rangle),\label{eq:leftmin}\end{equation}
up to $O(\theta)$ corrections, which are present if $x$ has the form $n_2(b_2-b_3)+b_2$ (and are small, for $\theta$ small). The computation leading to \eqref{eq:leftmin} neglected the presence of $\tau$ in the
definition of $u_S$ and was not based on an exact minimization of the sum in the right side of \eqref{mintobed}. However, the patient reader can check that, by performing an exact minimization, one can choose
$\tau$ of order $\theta$ and $\sigma$ can be chosen as in \eqref{eq:leftmin}, up to a bounded $O(\theta)$ correction (details left to the patient reader). Similarly, the exact minimization along the right boundary leads to a slip field equal to the opposite of \eqref{eq:leftmin}, up to a bounded $O(\theta)$ correction.

In conclusion, neglecting $O(\theta)$ fluctuation terms in the boundary slip field, which are not expected to contribute to the optimal grain energy per unit vertical length at the dominant order in the limit
of small $\theta$ and large $n$, the optimal slip field for the minimization problem \eqref{mintobed}
equals $\mp b_1(\langle 2\theta n_2\rangle+\langle \theta n_2\rangle)$ on the boundary bonds of the grain with vertical coordinate $n_2$ (the minus and plus signs are for the left end right boundaries, respectively).
Notice that the charge distribution $d\sigma$ of such a slip field consists of isolated charges equal to $\pm b_1$ or $\pm 2b_1$ on suitable faces of the left and right boundaries (with opposite signs on the two boundaries),
vertically separated  in average by a distance $\frac1{2\theta}$. The average density of such boundary charge distribution
equals $\pm b_1/(3\theta)$, the same as the charge distribution \eqref{wall_dist}, provided
we identify $m$ with $1/(3\theta)$. 

\section{The cellular complex of the FCC lattice}\label{app.FCC}

We recall that the three-dimensional FCC lattice $\mathcal L$ is the Bravais lattice with basis vectors $b_1,b_2,b_3$, as in \eqref{b123}. We also define the dual lattice $\mathcal L^*$ as the Bravais lattice with basis vectors $m_1,m_2,m_3$, defined as
 \begin{equation} m_1= \sqrt{2}\pi \begin{pmatrix}-1\\1\\1\end{pmatrix}, \quad m_2=\sqrt{2}\pi \begin{pmatrix}1\\-1\\1\end{pmatrix}, \quad m_3= \sqrt{2}\pi \begin{pmatrix}1\\1\\-1\end{pmatrix}.
 \label{eq:mi}\end{equation}
Note that $b_i\cdot m_j=2\pi \delta_{i,j}$, with $i,j=1,2,3$. For later reference, we also let $m_4=m_1+m_2+m_3= \sqrt{2}\pi \begin{pmatrix} 1 \\ 1 \\ 1 \end{pmatrix}$.
In terms of these definitions, the cellular complex associated with the FCC lattice is defined in terms of the following cells:
\begin{enumerate}
\item The vertices $x\in E_0$ are the vertices of $\mathcal L$, of the form $x=n_1b_1+n_2b_2+n_3b_3$.
\item The edges $e\in E_1$ are the ordered pairs of nearest neighbour vertices of $\mathcal L$, namely pairs $(x,x')$ with $x'-x=\pm b_l$, $l=1,\ldots, 6$: here $b_1,b_2,b_3$ are the same as \eqref{b123},
and we recall that $b_4= b_3-b_2$, $b_5= b_1-b_3$, $b_6=b_2-b_1$.
The action of the boundary operator on $E_1$ is defined by: $\partial (x_1,x_2) =\{x_1,x_2\}$, for any $(x_1,x_2)\in E_1$. Note that,
in the notation of Sect.\ref{subdec}, $\partial e=V(e)$, $\forall e\in E_1$, where $V(e)$ is the set of vertices of $e$.
\item The faces $f\in E_2$ can be identified with the 3-cycles of nearest-neighbor vertices $(x_1,x_2,x_3)$ such that $(x_i,x_j)\in E_1$, for $i\neq j$, $i,j=1,2,3$. There are 8 fundamental types of faces:
\begin{eqnarray*}
&f_1= (0, b_2,b_3), \qquad &f_5= (0, b_1, b_2),\\
&f_2= (0, -b_2, -b_3), \qquad &f_6= (0, -b_1, -b_2),\\
&f_3= (0, b_3, b_1), \qquad &f_7= (0, b_6,-b_5),\\
&f_4= (0, -b_3, -b_1), \qquad &f_8= (0, -b_6,b_5).
\end{eqnarray*}
plus those with opposite orientations:
\begin{eqnarray*}
&f_1'= (0, b_3,b_2),\qquad &f_5'= (0, b_2,b_1),\\
&f_2'= (0, -b_3,-b_2),\qquad &f_6'= (0, -b_2,-b_1),\\
&f_3'= (0, b_1,b_3),\qquad &f_7'= (0, -b_5,b_6),\\
&f_4'= (0, -b_1,-b_3),\qquad &f_8'= (0, b_5,-b_6).
\end{eqnarray*}
The orientation $o(f)$ of each face $f=(x_1,x_2,x_3)$ can be identified with the normal vector computed via the `right-hand rule', that is, $o((x_1,x_2,x_3))=
(x_2-x_1)\times (x_3-x_1)$.
Note, in particular, that the orientation of the fundamental faces $f_j, f_j'$ are: $o(f_j)=-o(f_j')=\frac1{\sqrt{6} \pi}m_{\lceil j/2\rceil}$.
The set $E_2$ can be obtained by translating the fundamental faces $\{f_1,\ldots, f_8'\}$ by the elements of $\LL$.
The action of the boundary operator on $E_2$ is defined by: $\partial (x_1,x_2,x_3) =\{(x_1,x_2), (x_2,x_3), (x_3,x_1)\}$, for any $(x_1,x_2,x_3)\in E_2$. The set of vertices of a face is simply $V((x_1,x_2,x_3))=\{x_1,x_2,x_3\}$.
For later reference, we also let $G(f)=\frac13\sum_{x\in V(f)}x$ be the baricenter of $f$.
\item The volumes $v\in E_3$ are the tetrahedra and the octahedra obtained by translating those shown in Fig.~\ref{fig.tile} by the elements of $\mathcal L$, together with an orientation $o(v)\in\{\pm\}$;
we shall refer to the positive orientation as to the `outward' orientation, and to the negative as to the `inward'.
Any element of $v\in E_3$ can be uniquely identified with the pair $(V(v), o(v))$, where $V(v)$ is the vertex set of $v$ (note, in fact, that the un-oriented volume associated to $v$ is the convex hull of $V(v)$).
The vertex sets of the r-tetrahedra in Fig.~\ref{fig.tile} are of the form $x+\{0,b_1,b_2,b_3\}$, with $x\in\mathcal L$.
The vertex sets of the g-tetrahedra in Fig.~\ref{fig.tile} are of the form $x+\{0,-b_1,-b_2,-b_3\}$, with $x\in\mathcal L$.
The vertex sets of the octahedra in Fig.~\ref{fig.tile} are of the form $x+\{b_1,b_2,b_3,b_1+b_2,b_1+b_3,b_2+b_3\}$. with $x\in\mathcal L$. For later reference, we also let $G(v)=\frac1{|V(v)|}\sum_{x\in V(v)}x$ be the baricenter of $v$.
The boundary operator on $E_3$ is defined by the condition that its action on $v\in E_3$ returns the faces of its boundary, with the outward orientation, if $o(v)=+$, and the inward orientation, if $o(v)=-$.
In formulae, $\partial v=\{f\in E_2:\ V(f)\subset V(c)\quad \text{and}\quad {\text{sign}}[(G(f)-G(v))\cdot o(f)]=o(v)\}$.
\end{enumerate}

\begin{figure}\begin{center}
\includegraphics[width=.45\textwidth]{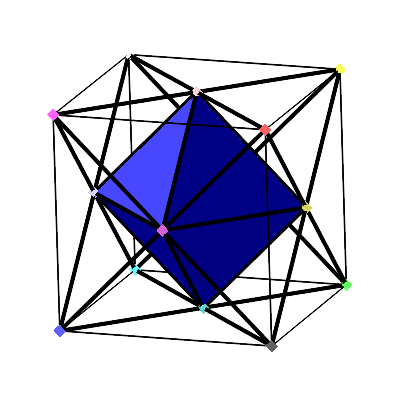}
\includegraphics[width=.45\textwidth]{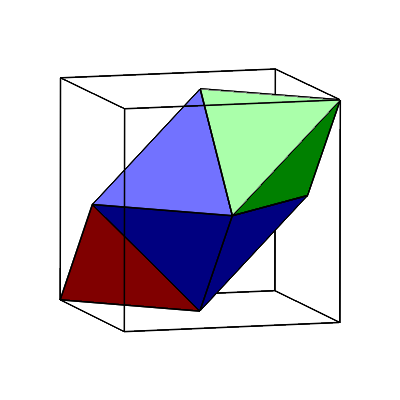}\end{center}
\caption{The left left panel shows the face centered cubic structure. Bonds are shown with bold lines. The edges of the cube are indicated with thin lines, they correspond to next-nearest neighbors.\\
The primitive unit cell of the FCC lattice is shown in the right panel. It can be dissected into a regular octahedron and two regular tetrahedra, which are the 3-cells of our cellular complex. The figure shows that there are two inequivalent type of tetrahedra: red and green.
We shall call r-tetrahedra (resp. g-tetrahedra) those that can be translated into the red (resp. green) tetrahedron.}
\label{fig.tile}
\end{figure}

\section{On the operator $A$ and its inverse}
\label{app.c0}

In this appendix we discuss and prove a few basic properties of the operator $\dd^*_0 B\dd_0$, both in the case that it acts on the $0$-forms associated with the infinite FCC lattice $\mathcal L$, and in the
case that it acts on those associated with a finite box $\Lambda{=\Lambda^{(N)}}\subset \mathcal L$, of the form \eqref{LN}, with Dirichlet boundary conditions. In order to avoid confusion between the two
cases, in this appendix (contrary to the rest of the paper) we denote by $\form^0$, resp. $\form^0_\Lambda$, the set of $0$-forms associated with the infinite lattice, resp. with the box $\Lambda$ with Dirichlet boundary conditions. Correspondingly, we denote by $A$, resp. $A_\Lambda$, the operator $\dd^*_0 B\dd_0$ acting on $\form^0$, resp. $\form^0_\Lambda$.
Note that  $A_\Lambda$ can be rewritten as \begin{equation}A_\Lambda=\sum_{l=1}^6\Pi_l \Delta_l,\label{b1}\end{equation} where $\Pi_l=b_l\otimes b_l$ is the projection along $b_l$  and $\Delta_l:\form^0_\Lambda\to \form^0_\Lambda$ is the (non-negative) one-dimensional Laplacian in the direction $b_l$, namely, if $f\in\Omega_0^\Lambda$, then $\Delta_l f(z)=2f(z)-f(z+b_l)-f(z-b_l)$.

\subsection{Invertibility of $A_\Lambda$}
Using \eqref{b1} and the fact that $\Delta_l\ge 0$, we find $A_\Lambda\ge \tilde A_\Lambda:=\sum_{l=1}^3 \Pi_l\Delta_l$. The operator $\tilde A_\Lambda$ acts diagonally on the $k$ index of the Dirichlet basis $\{u_{k,j}\}^{j=1,2,3}_{k\in \Lambda^*_D}$, where, if $m_1,m_2,m_3$ are the basis vectors of $\mathcal L^*$, see \eqref{eq:mi},
$$\Lambda_D^*:=\{k=k_1m_1+k_2m_2+k_3m_3: \ k_l=\scalebox{1.1}{$\frac{n_l}{2(N+1)}$}\ \ \text{with}\ \ n_l=1,\ldots, N\}$$
and
$$u_{k,j}(x)=\Big(\frac{2}{N+1}\Big)^{3/2}\Big[\prod_{l=1}^3\sin(2\pi k_l x_l)\Big]e_j,$$
with $k_l=\frac1{2\pi}k\cdot b_l$ and $e_j$ the $j$-th standard Euclidean basis vector. We have: $\tilde A_\Lambda u_{k,j}(x)=2\big[\sum_{l=1}^3 \alpha_l\Pi_l\big]_{ij} u_{k,i}(x)$,
where $\alpha_l:=1-\cos(2\pi k_l)$, which is positive for $k\in \Lambda_D^*$. Of course, $2\sum_{l=1}^3 \alpha_l\Pi_l\ge 2\min\{\alpha_1,\alpha_2,\alpha_3\}\sum_{l=1}^3 \Pi_l$.
By using the explicit form of $\Pi_l$, we get $2\sum_{l=1}^3\Pi_l=\begin{pmatrix} 2 & 1 & 1 \\ 1 & 2 & 1 \\ 1  & 1 & 2 \end{pmatrix}$, whose smallest eigenvalue is $1$, that is,
$2\sum_{l=1}^3\Pi_l\ge \mathds 1$.
In conclusion, $A_\Lambda\ge \tilde A_\Lambda\ge \min_{k\in\Lambda_D^*} (1-\cos(2\pi k_l) )\mathds 1$, which is positive, and, therefore, proves the invertibility of $A_\Lambda$ for any finite box $\Lambda$.

\subsection{Proof of \eqref{upb}}
In order to prove \eqref{upb}, we derive upper and lower bounds on $\langle g,A^{-1}_\Lambda g \rangle$, that is, the argument of the limit in the left side of \eqref{upb}, in the notation of this appendix.
For the reader's convenience, we recall that $g=g_{v_0;{x}}={{\bf 1}_x}v_0$, where ${x\in}\mathcal L$ and $v_0\in\mathcal L^*$. With no loss of generality (since we are interested in the thermodynamic limit $\Lambda{\to}\mathcal L$), we assume that $x\in\Lambda$. The important feature of $g$ to be used in the following is that it is compactly supported, with support contained in $\Lambda$. Note that
\begin{equation}\label{app.var}-\langle g,A^{-1}_\Lambda g\rangle=\min_{u\in\form^0_\Lambda} \big(\langle u, A u\rangle -2\langle u,g\rangle\big).\end{equation}
We recall that the minimum in the right side is over the compactly supported $0$-forms $u:\mathcal L\to\mathbb R^3$, whose support is contained in $\Lambda$. In order to get a lower bound, we
write the quadratic function $\langle u, A u\rangle -2\langle u,g\rangle$ in Fourier space, by using the convention $u(z)=\int_{\mathcal B}\frac{dk}{|\mathcal B|} \hat u(k) e^{-ikz}$, see the line after \eqref{upb} for the definition of $\mathcal B$; then, we complete the square and drop the non-negative $u$-dependent term, thus getting
\begin{equation}\label{app.var2}-\langle g,A^{-1}_\Lambda g\rangle\ge -\int_{\mathcal B} \frac{dk}{|\mathcal B|} \hat g(-k)\cdot \hat A^{-1}(k)\hat g(k),\end{equation}
with $\hat g(k)$ and $\hat A(k)$ defined as in \eqref{hatgk}-\eqref{eqAk}. As anticipated in Sect.\ref{sec.swme}, $\hat A^{-1}(k)$ is singular only at $k=0$, close to which it behaves like $\sim k^{-2}$, see below for a proof: therefore, the right side of \eqref{app.var2} is finite for any compactly supported $g$.

In order to get an upper bound, we use the test function $u_*(z):={\chi_\Lambda}(z)u_\infty(z)$, where {$\chi_\Lambda(z):=\min\{1,4\,\textrm{dist}(z,\Lambda^c)/N\}$ (recall that $N$ is the side of the box $\Lambda$, see 	\eqref{LN})} and 
\begin{equation} u_\infty(z)=\int_{\mathcal B} \frac{dk}{|\mathcal B|} \hat A^{-1}(k)\hat g(k) e^{-ik\cdot z},\label{uinfty}\end{equation}
thus getting, {for $N$ sufficiently large,}
\begin{equation}\label{app.var3}-\langle g,A^{-1}_\Lambda g\rangle\le \langle u_*, A u_*\rangle -2\langle u_*,g\rangle=\sum_{e\in E_1} \big(\dd u_*(e)\cdot \delta e\big)^2-2\langle u_\infty,g\rangle,
\end{equation}
where we recall that, for an edge $e=(x,y)$, $\delta e=y-x$, and in the last identity we used the fact that, {thanks to the definition of $\chi_\Lambda$, 
for $N$ sufficiently large} the support of $g={{\bf 1}_{x}}v_0$ is contained in ${\textrm{supp}{\bf 1}(\chi_\Lambda=1)\subset}\Lambda$, so that in particular
$u_*=u_\infty$ on the support of $g$. Moreover, {letting $u(z_e):=\frac12 (u(x)+u(y))$ for an edge $e=(x,y)$,}
\begin{equation}\label{app.var4}
\begin{split}\sum_{e\in E_1} \big(\dd u_*(e)\cdot \delta e\big)^2&\le \sum_{e\in E_1} \big(\dd u_\infty(e)\cdot \delta e\big)^2+{\sum_{e\in E_1}
(\dd \chi_\Lambda(e))^2\,(u_\infty(z_e)\cdot\delta e)^2}\\
&{\le \langle u_\infty, A
u_\infty\rangle+ \frac{C}{N^2}\sum_{\substack{z\in\mathcal L:\\ \textrm{dist}(z,\Lambda^c)\le N/4+2}}|u_\infty(z)|^2},\end{split}\end{equation}
{for some universal constant $C>0$.} 
Plugging \eqref{app.var4} in \eqref{app.var3}, and using the fact that $\langle u_\infty, A u_\infty\rangle-2\langle u_\infty,g\rangle$ is equal to the right side of \eqref{app.var2}, we find
\begin{equation}\label{app.var5}-\langle g,A^{-1}_\Lambda g\rangle\le -\int_{\mathcal B} \frac{dk}{|\mathcal B|} \hat g(-k)\cdot \hat A^{-1}(k)\hat g(k)+
{\frac{C}{N^2}}
\sum_{\substack{{z\in\mathcal L:}\\ {\textrm{dist}(z,\Lambda^c)\le N/4+2}}}|u_\infty(z)|^2.\end{equation}
We will prove below, see Sect.\ref{appb5}, that
\begin{equation}{|u_\infty(z)|=
\Big|\int_{\mathcal B} \frac{dk}{|\mathcal B|}\hat A^{-1}(k)v_0\, e^{ik\cdot (x-z)}\Big|\le c_1\frac{1+\log|x-z|}{|x-z|}},\label{cinfty}\end{equation}
for some positive constant $c_1$, so that, {for $N$ large enough and a suitable $C'>0$,}
\begin{equation}\label{ult}{\frac{C}{N^2}\sum_{\substack{z\in\mathcal L:\\ \textrm{dist}(z,\Lambda^c)\le N/4+2}}|u_\infty(z)|^2\le C'\frac{\log^2 N}N}. \end{equation}
Now, for $\Lambda=\Lambda^{(N)}$ large enough (see \eqref{LN}), the right side of \eqref{ult} vanishes as $N\to\infty$, which concludes the proof of \eqref{upb}.

\subsection{Invertibility of $\hat A(k)$ for $k\neq0$}
Let us rewrite $\hat A(k)$ in \eqref{eqAk} as
$\hat A(k)=2\sum_{l=1}^6 \alpha_l\Pi_{l}$, where $\alpha_l:=1-\cos(k\cdot b_l)\ge 0$. By using the explicit form of $\Pi_l$, following from the explicit expression of the vectors $b_1,\ldots, b_6$, we find:

\begin{equation} \hat A(k)=\begin{pmatrix}
\alpha_2+\alpha_3+\alpha_5+\alpha_6 & \alpha_3-\alpha_6 & \alpha_2-\alpha_5 \\
\alpha_3-\alpha_6 & \alpha_1+\alpha_3+\alpha_4+\alpha_6 &  \alpha_1-\alpha_4 \\
\alpha_2-\alpha_5 & \alpha_1-\alpha_4 & \alpha_1+\alpha_2+\alpha_4+\alpha_5  \end{pmatrix}, \end{equation}
whose determinant is
\begin{equation}\label{Akalpha}\begin{split}
\det \hat A(k)=4\big[&
\alpha_1\alpha_2\alpha_3+\alpha_1\alpha_2\alpha_4+\alpha_1\alpha_2\alpha_5+\alpha_1\alpha_3\alpha_4+\alpha_1\alpha_3\alpha_6+\alpha_1\alpha_4\alpha_5\\
+&\alpha_1\alpha_4\alpha_6+\alpha_1\alpha_5\alpha_6 +
\alpha_2\alpha_3\alpha_5+\alpha_2\alpha_3\alpha_6+\alpha_2\alpha_4\alpha_5\\ +&\alpha_2\alpha_4\alpha_6+\alpha_2\alpha_5\alpha_6+\alpha_3\alpha_4\alpha_5+\alpha_3\alpha_4\alpha_6+\alpha_3\alpha_5\alpha_6\big].\end{split}\end{equation}
Note that all the terms in the sum are non-negative, bacause $\alpha_l\ge 0$. We want to argue that $(\alpha_1,\alpha_2,\alpha_3)\neq (0,0,0)\Rightarrow \det\hat A(k)\neq 0$. Recall that
\begin{equation} \label{defbl} b_4= b_3-b_2,\quad b_5= b_1-b_3,\quad b_6=b_2-b_1,\end{equation}
so that $\alpha_2=\alpha_3=0\Rightarrow \alpha_4=0$, etc.

If $\alpha_1,\alpha_2,\alpha_3$ are all positive, then $\det\hat A(k)>0$, simply because the first term in the right side of \eqref{Akalpha} is positive.

Suppose now that two of the elements of the triple
$(\alpha_1,\alpha_2,\alpha_3)$ are positive and third is zero, say $\alpha_1,\alpha_2>0$ and $\alpha_3=0$ (the other cases are treated analogously);
from \eqref{defbl}, it follows that $\alpha_4,\alpha_5>0$. Therefore,
$\det\hat A(k)>0$, because the factor $\alpha_1\alpha_2\alpha_4$, among others, is positive.

Finally, suppose that one of the elements of the triple
$(\alpha_1,\alpha_2,\alpha_3)$ is positive and the other two are zero, say $\alpha_1>0$ and $\alpha_2=\alpha_3=0$ (the other cases are treated analogously);
from \eqref{defbl}, it follows that $\alpha_5,\alpha_6>0$. Therefore, $\det\hat A(k)>0$, because the factor $\alpha_1\alpha_5\alpha_6$ is positive.

This completes the proof that $\hat A(k)$ is invertible iff $k\neq 0$ mod $\mathcal L^*$.

\subsection{Proof of \eqref{c0}}
By expanding $\hat A(k)$ in Taylor series in $k$ around $k=0$, we get
$$\hat A(k)=\sum_{l=1}^6 (k\cdot b_l)^2\Pi_{l}+O(k^3)\equiv \hat A_0(k)+O(k^3).$$
By using the explicit expression of the projectors $\Pi_l$, we find
\begin{eqnarray}\hat A_0(k)&=&\frac14\Big[(k_2+k_3)^2\begin{pmatrix} 0 & 0 & 0 \\ 0 &1 & 1\\ 0 & 1 & 1\end{pmatrix}+(k_2-k_3)^2\begin{pmatrix} 0 & 0 & 0 \\ 0 &1 & -1\\ 0 & -1 & 1\end{pmatrix}+\,\text{permutations}\Big]\nonumber\\
&=& \frac{k^2}{2}{\mathds 1}+\begin{pmatrix} k_1^2/2 & k_1k_2 & k_1k_3 \\ k_1k_2 &k_2^2/2 & k_2k_3\\ k_1k_3 & k_2k_3 & k_3^2/2\end{pmatrix}=\frac{k^2}{2}\mathds 1+ k\otimes k -\frac12\text{diag}(k_1^2,k_2^2,k_3^2),\end{eqnarray}
from which the upper bound in \eqref{c0} follows. We now get a lower bound on the eigenvalues of $\hat B_0(k):=\hat A_0(k)-\frac{k^2}2\mathds 1$. The characteristic polynomial of $\hat B_0(k)$ is
$$P(\lambda)=-\lambda^3+\frac{k^2}2\lambda^2+\frac34\lambda(k_1^2k_2^2+k_1^2k_3^2+k_2^2k_3^2)+\frac58k_1^2k_2^2k_3^2,$$
which has three real roots. It is easy to see that the smallest root is larger than $-ak^2$, with $a=\frac{\sqrt5-1}{4}$. This immediately follows from the fact that $P(-ak^2)\ge 0$ and $P'(\lambda)\le 0$, $\forall \lambda\le -ak^2$. In order to check the first of these two inequalities, note that
$$P(-ak^2)\ge k^6\big(a^3+\frac{a^2}{2}-\frac{a}4), $$
simply because $k^{-4}(k_1^2k_2^2+k_1^2k_3^2+k_2^2k_3^2)\le \frac13$, for all $k\neq 0$. Moreover, recalling that $a=(\sqrt5-1)/4$, we find that $a^3+\frac{a^2}{2}-\frac{a}4=0$, which implies $P(-ak^2)\ge0$. Finally,
in order to see that $P'(\lambda)\le 0$, $\forall \lambda\le -ak^2$, note that, if $\lambda\le -ak^2$, then
$$P'(\lambda)=-3\lambda^2+k^2\lambda+\frac34 (k_1^2k_2^2+k_1^2k_3^2+k_2^2k_3^2)\le k^4 \Big(-3a^2-a+\frac34 \frac{k_1^2k_2^2+k_1^2k_3^2+k_2^2k_3^2}{k^4}\Big). $$
Using again the fact that $k^{-4}(k_1^2k_2^2+k_1^2k_3^2+k_2^2k_3^2)\le \frac13$, we find that, for all $\lambda\le -a k^2$,
$$P'(\lambda)\le k^4 \Big(-3a^2-a+\frac14\Big),$$
which is negative for $a=(\sqrt5-1)/4$. In conclusion, $\hat B_0(k)=\hat A_0(k)-\frac{k^2}2\mathds 1\ge -a k^2$, from which the lower bound in \eqref{c0} follows.

\subsection{Proof of \eqref{logbbis} and \eqref{cinfty}}\label{appb5} {Both \eqref{logbbis} and \eqref{cinfty} follow from
\begin{equation}\label{eppero} \Big|\int_{\mathcal B}\frac{dk}{|\mathcal B|}(\hat A^{-1}(k)v_0)_{l}\, e^{-ik\cdot x}\Big|\le (\textrm{const.})\frac{1+\log|x|}{|x|},\end{equation}
which is valid for any $l\in\{1,2,3\}$. In order to prove \eqref{eppero}, we assume that $|x|\ge \epsilon^{-1}$} for an arbitrary, sufficiently small, $\epsilon$, and multiply the left side by $|x_j|$, with $j\in\{1,2,3\}$. 
Then, we rewrite it as:
\begin{eqnarray}\Big|x_j \int_{\mathcal B}\frac{dk}{|\mathcal B|}{(}\hat A^{-1}(k)v_0{)_l}\, e^{-ik\cdot x}\Big|&=& \Big| \int_{\mathcal B}\frac{dk}{|\mathcal B|}{(}\hat A^{-1}(k)v_0{)_l} \partial_{k_j}e^{-ik\cdot x}\Big|\nonumber\\
&=&\Big| \int_{\mathcal B}\frac{dk}{|\mathcal B|}{(}\partial_{k_j}\hat A^{-1}(k)v_0{)_l} e^{-ik\cdot x}\Big|.\label{eq95}\end{eqnarray}
Note that $\partial_{k_j}\hat A^{-1}(k)=-\hat A^{-1}(k)\cdot \partial_{k_j}\hat A(k) \cdot\hat A^{-1}(k)$, with $\partial_{k_j}\hat A(k)=2\sum_{l=1}^6 \Pi_{l}(b_l)_j\sin(k\cdot b_l)$. Recalling that $\hat A(k)$ is even, it is singular iff $k=0$, and, for $k$ close to zero, it can be bounded from above and below by (const.)$k^2$, we find that $\partial_{k_j}\hat A^{-1}(k)$ is odd, it is singular iff $k=0$ and, close to the singularity, it can be bounded from above by (const.)$|k|^{-3}$.
Therefore, the right side of \eqref{eq95} can be rewritten as
$$\Big| \int_{\mathcal B}\frac{dk}{|\mathcal B|}{(}\partial_{k_j}\hat A^{-1}(k)v_0{)_l} \sin(k\cdot x)\Big|$$
and, in order to bound it from above, we multiply the integrand by $1=\chi(k)+(1-\chi(k))$, where $\chi(k)$ is a positive, monotone, $C^\infty$ radial function, equal to $1$ for $|k|\le \epsilon$ and equal to $0$ for $|k|\ge 2\epsilon$.
Now, the term associated with $(1-\chi(k))$ is the Fourier transform of a $C^\infty$ function and, therefore, it decays faster than any power in real space. The term associated with
$\chi(k)$ can be bounded as follows:
$$({\rm const.})\Big( \int_{|k|\le |x|^{-1}}dk\, \frac{|\sin(k\cdot x)|}{|k|^3}+ \int_{ |x|^{-1}\le |k|\le 2\epsilon}dk\, \frac{1}{|k|^3}\Big)\le ({\rm const.})\log|x|. $$
Putting things together, we obtain that, if $|x|\ge \epsilon^{-1}$, then
$$ |x_j|\,\Big| \int_{\mathcal B}\frac{dk}{|\mathcal B|}{(}\hat A^{-1}(k)v_0{)_l} e^{-ik\cdot (x-y)}\Big|\le({\rm const.})\log|x|.$$
Summing over $j$ from $1$ to $3$, we get the desired estimate, {\eqref{eppero}.}

\section{Two technical estimates on cluster expansion}\label{app.CE}

\subsection{Proof of \eqref{CE.1}}
Starting from the definition of $z(\beta,q)$, eq.\eqref{zbetaq}, and using the bound on $\lambda(x)$ stated one line after \eqref{qf},
\begin{equation} |z(\beta,q)|\le\sum_{n\ge 1}\sum_{\substack{q_1,\ldots, q_n\in\form^2_*:\\ q_1+\cdots+q_n=q}}  \Big[\prod_{i=1}^n\mathds 1_{X_i\,{\text{is}\, {\text{connected}}}}\Big]
\Big[\prod_{i=1}^n\Big(\prod_{f\in X_i} e^{-\frac\beta2 w_0|q_i(f)|^2}\Big)\Big]\big|\varphi(X_1,\ldots,X_n)\big|,\label{zbetaq.bis}\end{equation}
We now split the exponential factor in two parts, $e^{-\frac\beta2 w_0|q_i(f)|^2}=e^{-\frac\beta4 w_0|q_i(f)|^2}e^{-\frac\beta4 w_0|q_i(f)|^2}$, and bound the product of $e^{-\frac\beta4 w_0|q_i(f)|^2}$
as
\begin{equation}\label{bl1}\prod_{i=1}^n\Big(\prod_{f\in X_i} e^{-\frac\beta4 w_0|q_i(f)|^2}\Big)\le e^{-\frac\beta4 w_0\|q\|_1},\end{equation}
where we used that (recall that $q=\sum_i q_i$)
$$\sum_i\sum_{f\in X_i}|q_i(f)|^2\ge \sum_i\sum_{f\in X_i}|q_i(f)|\ge\sum_{f\in \cup_i X_i}|\sum_i q_i(f)|\equiv \|q\|_1.$$
If we plug \eqref{bl1} in \eqref{zbetaq.bis} and then weaken the constraint $q_1+\cdots q_n=q$ into $\cup_i X_i=\supp(q)$, we obtain
\begin{equation} |z(\beta,q)|\le e^{-\frac\beta4 w_0\|q\|_1}\sum_{n\ge 1}\sum_{\substack{X_1,\ldots,X_n\ {\rm connected}:\\ \cup_iX_i=\supp(q)}}
\zeta(X_1)\cdots\zeta(X_n)\big|\varphi(X_1,\ldots,X_n)\big|,\label{zbetaq.tris}\end{equation}
with
\begin{equation}\label{defzeta}\zeta(X):=
\Big(\sum_{\substack{b\in\mathcal L:\\ b\neq 0}}e^{-\frac\beta4 w_0|b|^2}\Big)^{|X|}.\end{equation}
Eq.\eqref{zbetaq.tris} can be further bounded from above as
\begin{eqnarray} |z(\beta,q)|&\le& e^{-\frac\beta4 w_0\|q\|_1}\sum_{X_1\, {\rm connected}}\zeta(X_1)\big(1-\delta(X_1,\supp(q))\big)\cdot\label{zbetaq.tetris}\\
&\cdot& \Big[1+\sum_{n\ge 2}\sum_{\substack{X_2,\ldots,X_n\\ {\rm connected}}}\zeta(X_2)\cdots\zeta(X_n)\big|\varphi(X_1,\ldots,X_n)\big|\Big].\nonumber\end{eqnarray}
Now, if $a'(X)$ is such that $\sum_{X\,{\rm connected}}\zeta(X)e^{a'(X)}(1-\delta(X,X_*))\le a'(X_*)$
for any fixed, connected, non-empty $X_*$, then the sum in square brackets in the second line is bounded from above by $e^{a'(X)}$, see \cite[Theorem 5.4]{FV}.
In our case, if $\beta$ is sufficiently large, thanks to the definition of $\zeta(X)$, eq.\eqref{defzeta}, we can choose $a'(X)=e^{-\beta w_0/8} |X|$. Therefore,
\begin{equation}\begin{split} |z(\beta,q)|&\le e^{-\frac\beta4 w_0\|q\|_1}\sum_{X_1\, {\rm connected}}\zeta(X_1)e^{a'(X_1)}\big(1-\delta(X_1,\supp(q))\big)\nonumber\\
&\le e^{-\frac\beta4 w_0\|q\|_1}a'(\supp(q))=
e^{-\frac\beta4 w_0\|q\|_1}e^{-\frac{\beta}{8}w_0}|\supp(q)|,\end{split}\label{zbetaq.pentatris}\end{equation}
which is the desired estimate.

\subsection{Proof of \eqref{CE.2}}
Plugging \eqref{CE.1} in the left side of \eqref{CE.2}, and using the fact that $|\supp(q)|\le \|q\|_1$ and $\|q\|_2\le \|q\|_1$, we find
\begin{equation}\label{CE.2bis}\sum_{q\in\form^2_*}|z(\beta,q)|\, \|q\|_2^8\ \mathds 1(B(q)\ni e,e')\le e^{-\frac{\beta}{8}w_0}\sum_{q\in\form^2_*}
e^{-\frac\beta4 w_0\|q\|_1}\|q\|_1^9 \mathds 1(B(q)\ni e,e'). \end{equation}
We now weaken the constraint that $B(q)\ni e,e'$ into $\|q\|_1\ge {\rm dist}(e,e')$ and find that, for $\beta$ large enough,
\begin{equation}\begin{split}&
\sum_{q\in\form^2_*}|z(\beta,q)| \|q\|_2^8 \mathds 1(B(q)\ni e,e')\le \\ &\le e^{-\frac{\beta}{8}w_0}\sum_{q:\, \|q\|_1\ge {\rm dist}(e,e')}
e^{-\frac\beta4 w_0\|q\|_1}\|q\|_1^9 \le e^{-\frac{\beta}{8}w_0(1+{\rm dist}(e,e'))},\end{split}\end{equation}
as desired.

\bigskip

\bigskip
\footnotesize
\noindent\textit{Acknowledgments.}
This work has been supported by the European Research Council (ERC) under the European Union's Horizon 2020 research and innovation programme
(ERC CoG UniCoSM, grant agreement n.724939).
The authors wish to thank Roland Bauerschmidt and Adriana Garroni for very helpful discussions during early stages of this project.
A.G. would like to thank Massimiliano Pontecorvo for useful comments and references on discrete exterior calculus, {and Robin Reuvers 
for useful comments on the definition of the observable for testing translational symmetry breaking and on the null directions of the lattice Green's function.}


\begin{thebibliography}{SK}




\normalsize
\baselineskip=17pt


\bibitem{AKM16}
Adams, S., Koteck{\'y}, R., M{\"u}ller, S.:
Strict convexity of the surface tension for non-convex potentials.
arXiv:1606.09541.

\bibitem{ao05}
Ariza, M. P., Ortiz, M.:
Discrete crystal elasticity and discrete dislocations in crystals.
Arch. Rat. Mech. Anal. \textbf{178}, 149--226 (2005)

\bibitem{AO10}
Ariza, M. P., Ortiz, M.:
Discrete dislocations in graphene.
J. Mech. Phys. Sol. \textbf{58}, 710--734 (2010)

\bibitem{AOS10}
Ariza, M. P., Ortiz, M., Serrano, R.:
Long-term dynamic stability of discrete dislocations in graphene at
finite temperature. Int. J. Fract. \textbf{166}, 215--223 (2010)

\bibitem{Au15}
Aumann, S.:
Spontaneous breaking of rotational symmetry with arbitrary defects
and a rigidity estimate.
J. Stat. Phys. \textbf{160}, 168--208 (2015)

\bibitem{BCHMR18}
Bauerschmidt, R., Conache, D.,  Heydenreich, M., Merkl, F., Rolles,  S.:
Dislocation lines in three-dimensional solids at low temperature.
Ann. Henri Poincar\'e \textbf{20}, 3019--3057 (2019)

\bibitem{BoTu82}
Bott, R., Tu, L.:
Differential {F}orms in {A}lgebraic {T}opology.
Springer (1982)

\bibitem{BPT14}
Bourne, D., Peletier, M., Theil, F.:
Optimality of the triangular lattice for a particle system with
{W}asserstein interaction.
Comm. Math. Phys. \textbf{329}, 117--140 (2014)

\bibitem{BM99}
Brydges, D. C., Martin, Ph. A.:
Coulomb systems at low density: A review.
J. Stat. Phys. \textbf{96}, 1163--1330 (1999)

\bibitem{BC06}
Bulatov, V., Cai, W.:
Computer Simulations of Dislocations, volume 3.
Oxford University Press (2006)

\bibitem{CGM11}
Conti, S., Garroni, A., M\"uller, S.:
Singular kernels, multiscale decomposition of microstructure, and
dislocation models.
Arch. Rat. Mech. Anal. \textbf{199}, 779--819 (2011)

\bibitem{EL09}
E, W., Li, D.:
On the crystallization of 2d hexagonal lattices.
Comm. Math. Phys. \textbf{286}, 1099--1140 (2009)

\bibitem{FPP18}
Fanzon, S., Palombaro, M., Ponsiglione, M.:
Derivation of linearised polycrystals from a two-dimensional system
of edge dislocations.
SIAM J. Math. Anal. \textbf{51}, 3956--3981 (2019)

\bibitem{FT15}
{Flatley}, L., {Theil}, F.:
Face-centered cubic crystallization of atomistic configurations.
Arch. Rat. Mech. Anal. \textbf{218}, 363--416 (2015)

\bibitem{FV}
Friedli, S., Velenik, Y.:
Statistical Mechanics of Lattice Systems - A Concrete
Mathematical Introduction.
Cambridge University Press (2017)

\bibitem{FP78}
Fr\"ohlich, J., Park, Y.~M.:
Correlation inequalities and the thermodynamic limit for classical
and quantum continuous systems.
Comm. Math. Phys. \textbf{59}, 235--266 (1978)

\bibitem{FP81}
Fr\"ohlich, J., Pfister, C.-E.:
On the absence of spontaneous symmetry breaking and of crystalline
ordering in two-dimensional systems.
Comm. Math. Phys. \textbf{81}, 277--298 (1981)

\bibitem{FS81}
Fr\"ohlich, J., Spencer, T.:
The {K}osterlitz-{T}houless transition in two-dimensional {A}belian
spin systems and the {C}oulomb gas.
Comm. Math. Phys. \textbf{81}, 527--602 (1981)

\bibitem{fs82}
Fr\"ohlich, J., Spencer, T.:
Massless phases and symmetry restoration in {A}belian gauge theories
and spin systems.
Comm. Math. Phys. \textbf{83}, 411--454 (1982)

\bibitem{G19}
Ga\'al, A.~T.:
Long-range orientational order of a random near lattice hard sphere
and hard disk process.
J. Appl. Probab. \textbf{57}, 559--577 (2020)

\bibitem{GLM10}
Garroni, A., Leoni, G., Ponsiglione, M.:
Gradient theory for plasticity via homogenization of discrete
dislocations.
J. Eur. Math. Soc. (JEMS) \textbf{12}, 1231--1266 (2010)

\bibitem{GM05}
Garroni, A., M\"uller, S.:
{$\Gamma$}-limit of a phase-field model of dislocations.
SIAM J. Math. Anal. \textbf{36}, 1943--1964 (2005)

\bibitem{GM06}
Garroni, A., M\"{u}ller, S.:
A variational model for dislocations in the line tension limit.
Arch. Ration. Mech. Anal. \textbf{181}, 535--578 (2006)

\bibitem{Hat02}
Hatcher, A.:
Algebraic Topology.
Cambridge University Press (2002)

\bibitem{HMR14}
Heydenreich, M., Merkl, F., Rolles, S.:
Spontaneous breaking of rotational symmetry in the presence of defects.
Electron. J. Probab. \textbf{19} (2014)

\bibitem{HL82}
Hirth, J., Lothe, J.:
Theory of {D}islocations.
Krieger Publishing Company (1982)

\bibitem{H18}
{Hudson}, T.:
An existence result for Discrete Dislocation Dynamics in three dimensions.
arXiv:1806.00304.

\bibitem{ISV}
Ioffe, D., Shlosman, S., Velenik, Y.:
2d models of statistical physics with continuous symmetry: The case of singular interactions.
Comm. Math. Phys. \textbf{226}, 433--454 (2002)

\bibitem{KK86}
Kennedy, T., King, C.:
Spontaneous symmetry breakdown in the {A}belian {H}iggs model.
Comm. Math. Phys. \textbf{104}, 327--347 (1986)

\bibitem{KT}
Kosterlitz, J.~M., Thouless, D.~J.:
Ordering, metastability and phase transitions in two-dimensional systems.
J. Phys. C: Solid State Phys. \textbf{6}(7), 1181--1203 (1973)

\bibitem{LL17}
{Lauteri}, G., {Luckhaus}, S.:
An energy estimate for dislocation configurations and the emergence of Cosserat-type structures in metal plasticity.
arXiv:1608.06155.

\bibitem{LL72}
Lieb, E.~H., Lebowitz, J.:
The constitution of matter: existence of thermodynamics for systems composed of electrons and nuclei.
Adv. Math. \textbf{9}, 316--398 (1972)

\bibitem{AM16}
Mendez, J., Ariza, M. P.:
Harmonic model of graphene based on a tight binding interatomic potential.
J. Mech. Phys. Sol. \textbf{93}, 198--223 (2016)

\bibitem{M68}
{Mermin}, N. D.: 
Crystalline order in two dimensions.
Phys. Rev. \textbf{176}, 250--254 (1968)

\bibitem{MW66}
{Mermin}, N. D., {Wagner}, H.:
Absence of ferromagnetism or antiferromagnetism in one- or two-dimensional isotropic {H}eisenberg models.
Phys. Rev. Lett. \textbf{17}, 1133--1136 (1966)

\bibitem{NH}
Nelson, D.~R., Halperin, B.~I.:
Dislocation-mediated melting in two dimensions.
Phys. Rev. B \textbf{19}, 2457--2484 (1979)

\bibitem{Pf81}
Pfister, C.-E.:
On the symmetry of the {G}ibbs states in two dimensional lattice systems.
Comm. Math. Phys. \textbf{79}, 181--188 (1981)

\bibitem{RS50}
Read, W., Shockley, W.:
Dislocation models of crystal grain boundaries.
Phys. Rev. \textbf{78}, 275--289 (1950)

\bibitem{Rich}
Richthammer, T.:
Translation-invariance of two-dimensional {G}ibbsian point processes.
Comm. Math. Phys. \textbf{274}, 81--122 (2007)

\bibitem{SB06}
Sutton, A., Balluffi, R.:
Interfaces in Crystalline Materials.
Oxford University Press (2006)

\bibitem{T06}
Theil, F.:
A proof of crystallization in two dimensions.
Comm. Math. Phys. \textbf{262}, 209--236 (2006)

\bibitem{Y}
Young, A.~P:
Melting and the vector {C}oulomb gas in two dimensions.
Phys. Rev. B \textbf{19}, 1855--1866 (1979)

\end{thebibliography}
\end{document}